\documentclass{amsart}

\usepackage{fancyhdr}
\usepackage{mathscinet}
\usepackage{graphicx, xcolor}
\usepackage{bm}	
\usepackage{amsmath}
\usepackage{amsthm}	
\usepackage{amsfonts}	
\usepackage{amssymb}
\usepackage{mathrsfs}
\usepackage{latexsym}
\usepackage{enumerate}
\usepackage{dsfont}
\usepackage{amssymb}
\usepackage{yhmath}
\usepackage{mathtools}
\usepackage{array}
\usepackage[all, knot]{xy}
\usepackage[super]{nth}
\usepackage[all]{xy}
\input xy
\xyoption{all}
\xyoption{2cell}
\xyoption{v2}
\usepackage{hyperref}

\DeclareMathOperator{\Aut}{Aut}

\DeclareMathOperator{\End}{End}
\DeclareMathOperator{\tr}{tr}
\DeclareMathOperator{\pr}{pr}

\DeclareMathOperator{\im}{im}

\DeclareMathOperator{\id}{id}
\DeclareMathOperator{\rank}{rank}
\DeclareMathOperator{\hol}{hol}

\DeclareMathOperator{\ad}{ad}

\DeclareMathOperator{\vect}{Vect}

\DeclareMathOperator{\ovect}{\Omega Vect}
\DeclareMathOperator{\struct}{Struct}
\DeclareMathOperator{\Iso}{Iso}

\DeclareMathOperator{\Hom}{Hom}

\DeclareMathOperator{\Der}{Der}
\DeclareMathOperator{\Gr}{Gr}
\theoremstyle{plain}
\newtheorem{theorem}{Theorem}[section]
\newtheorem{corollary}[theorem]{Corollary}
\newtheorem{lemma}[theorem]{Lemma}
\newtheorem{proposition}[theorem]{Proposition}

\theoremstyle{definition}
\newtheorem{definition}[theorem]{Definition}

\theoremstyle{remark}
\newtheorem{example}[theorem]{Example}

\newtheorem{remark}[theorem]{Remark}

\numberwithin{equation}{section}
\numberwithin{figure}{section}

\renewcommand{\cR}{\mathcal{R}}
\newcommand{\cS}{{\mathcal S}}

\newcommand{\cK}{{\mathcal K}}

\renewcommand{\cL}{{\mathcal L}}

\newcommand{\cF}{{\mathcal{F}}}
\newcommand{\CC}{{\mathbb C}}
\newcommand{\RR}{{\mathbb R}}
\newcommand{\ZZ}{{\mathbb Z}}

\newcommand{\sQ}{{\mathsf Q}}
\newcommand{\sP}{{\mathsf P}}
\newcommand{\sE}{{\mathsf E}}
\newcommand{\sF}{{\mathsf F}}
\newcommand{\sA}{{\mathsf A}}

\newcommand{\sS}{{\mathsf S}}
\newcommand{\sR}{{\mathsf R}}
\newcommand{\sL}{{\mathsf L}}
\newcommand{\sH}{{\mathsf H}}
\newcommand{\rk}{{\mathsf {rank}\,}}
\newcommand{\ev}{{\mathsf {ev}}}
\newcommand{\de}{{\mathsf {d}}}

\renewcommand{\a}{\alpha}
\renewcommand{\b}{\beta}
\renewcommand{\c}{\gamma}

\renewcommand{\d}{\partial}
\newcommand{\lo}{\longrightarrow}
\newcommand{\sone}{{S^1}}
\newcommand{\x}{\times}

\newcommand{\CS}{{C\mspace{-1mu}S}}
\newcommand{\Ch}{{C\mspace{-1mu}h}}
\newcommand{\I}{{[0,1]}}
\newcommand{\deR}{{\mathrm{deR}}}

\newcommand{\dK}{\check{\mathcal{K}}^{-1}}
\newcommand{\dL}{\check{\mathcal{L}}^{-1}}
\newcommand{\edK}{\check{\mathcal{K}}^{0}}
\newcommand{\simto}{\xrightarrow{\;\;\cong\;\;}}
\renewcommand{\t}{\triangle}
\newcommand{\oast}{\circledast}

\newcommand{\ul}[1]{\underline{{#1}}}

\newcommand{\lan}{\langle}
\newcommand{\ran}{\rangle}
\newcommand{\llan}{\lan\mspace{-5mu}\lan}
\newcommand{\rran}{\ran\mspace{-5mu}\ran}

\begin{document}

\title[A Geometric Model for Odd Differential $K$-theory]{A Geometric Model for Odd Differential $K$-theory}
 \author[P. Hekmati]{Pedram Hekmati}
  \address[P. Hekmati]
  {School of Mathematical Sciences\\
  University of Adelaide\\
  Adelaide, SA 5005 \\
  Australia}
  \email{pedram.hekmati@adelaide.edu.au}

 \author[M. K. Murray]{Michael K. Murray}
  \address[M. K. Murray]
  {School of Mathematical Sciences\\
  University of Adelaide\\
  Adelaide, SA 5005 \\
  Australia}
  \email{michael.murray@adelaide.edu.au}
  
  \author[V. S. Schlegel]{Vincent S. Schlegel}
\address[V. S. Schlegel]
{Institut f\"{u}r Mathematik\\
Universit\"{a}t Z\"{u}rich\\
Winterthurerstrasse 190\\
8057 Z\"{u}rich \\
Switzerland}
\email{vincent.schlegel@math.uzh.ch}

 \author[R. F. Vozzo]{Raymond F. Vozzo}
  \address[R. F. Vozzo]
{School of Mathematical Sciences\\
  University of Adelaide\\
  Adelaide, SA 5005 \\
  Australia}
\email{raymond.vozzo@adelaide.edu.au}

\date{\today}

\thanks{The authors acknowledge support under the Australian Research Council's Discovery funding schemes: DP120100106, DP130102578 and DE120102657.}


\begin{abstract}
Odd $K$-theory has the interesting property that it admits an infinite number of inequivalent differential refinements. In this paper we provide a bundle theoretic model for odd differential $K$-theory using the  caloron correspondence and prove that this refinement is unique up to a unique natural isomorphism.  We characterise the odd Chern character and its transgression form in terms of a connection and Higgs field and discuss some applications. Our model can be seen as the odd counterpart to the Simons-Sullivan construction of even differential $K$-theory. We use this model to prove a conjecture of Tradler--Wilson--Zeinalian \cite{TWZ}, which states that the model developed there also defines the unique differential extension of odd $K$-theory.

\end{abstract} 
\maketitle

\tableofcontents


\renewcommand{\theequation}{\emph{i}\!\! .\arabic{equation}}
\section*{Introduction}
Since their inception in the guise of the differential characters of Cheeger--Simons \cite{ChS1}, differential cohomology theories have become an increasingly important tool in mathematics and mathematical physics.
Put simply, differential cohomology theories are refinements of generalized (Eilenberg--Steenrod) cohomology theories that naturally include extra differential form data.
Using abstract formalism, Hopkins and Singer constructed in \cite{HS} a differential cohomology theory associated to any given generalized cohomology theory.
Bunke and Schick developed in \cite{BS1} an axiomatic characterisation for differential extensions and later proved \cite{BS3} that, under certain conditions, such extensions are unique up to unique isomorphism.

Recently there has been a vigorous discussion on the properties and applications of differential extensions, centering particularly on differential extensions of topological $K$-theory.
There is already a variety of different models for differential $K$-theory appearing in the literature \cite{BS3, FL,HS}.
As is the case with ordinary $K$-theory, the group structure of differential $K$-theory splits into odd and even degree parts.
The Bunke--Schick uniqueness results are enough to guarantee that any two differential extensions of even $K$-theory are isomorphic, however this is not the case for odd $K$-theory where extra data is required to obtain uniqueness.

A primary consideration when constructing differential extensions is to obtain an intuitive geometric model that allows for straightforward calculations.
A particularly intuitive model for even differential $K$-theory is provided by Simons and Sullivan \cite{SSvec}.
The Simons--Sullivan model uses \emph{structured vector bundles}---smooth vector bundles equipped with an equivalence class of connections defined by Chern--Simons exactness---to incorporate differential form data into ordinary even $K$-theory.
This model has an obvious geometric appeal and its presentation avoids the additional differential forms appearing in other models, such as \cite{FL}.

In this paper, we introduce a natural odd-degree counterpart to the Simons--Sullivan model and show that it defines odd differential $K$-theory.
This bundle-theoretic construction is achieved using \emph{structured $\Omega$ vector bundles}.
Structured $\Omega$ vector bundles, or more simply \lq\lq $\Omega$ bundles'' are best viewed as the result of applying a sort of smooth suspension to the structured vector bundles of Simons--Sullivan.
The advantage of thinking in terms of $\Omega$ bundles is that it allows one to understand the odd differential $K$-theory of a manifold $M$ through differential-geometric structures living on $M$ as opposed to the suspension $\Sigma M$.
Ultimately the choice  to use $\Omega$ bundles is mostly aesthetic and everything could be done using regular bundles over $\Sigma M$, however we find that the use of $\Omega$ bundles simplifies certain arguments and emphasises the link to the Simons--Sullivan model.

In concrete terms, $\Omega$ bundles are Fr\'{e}chet vector bundles with the additional property that each fibre is a free finitely generated $L\CC$-module, with $L\CC$ the ring of smooth loops in $\CC$.
Equivalently, $\Omega$ bundles may be viewed as the vector bundle objects naturally associated to principal $\Omega GL(n)$-bundles via the associated bundle construction.
Here $\Omega GL(n)$ is the Fr\'{e}chet Lie group of smooth loops in $GL(n)$ based at the identity

A central tool in this work is the \emph{caloron correspondence} of \cite{MV, V}, whereby an $\Omega GL(n)$-bundle $\sQ \to M$ determines a framed $GL(n)$-bundle\footnote{recall (cf.  \cite[Definition 3.1]{MV}) that a $G$-bundle $Q\to M\x\sone$ is \emph{framed} if it is equipped with a choice of section $s \in \Gamma(M\x\{0\},P)$.} $Q \to M\x\sone$ and, conversely, a framed $GL(n)$-bundle $P \to M\x\sone$ determines an $\Omega GL(n)$-bundle $\sP \to M$.
The assignments $\sQ \mapsto Q$ and $P \mapsto \sP$ are functorial and are called the \emph{caloron transform} and \emph{inverse caloron transform} respectively.
These functors give an equivalence of categories and a key property of this equivalence is that it persists at the level of connective data.
Briefly, recall that a Higgs field for the $\Omega GL(n)$-bundle  $\sQ \to M$ is a map $\Phi \colon \sQ\to L\mathfrak{gl}(n)$ satisfying the condition
\begin{equation}
\label{eqn:twistedeq}
\Phi(q\c) = \ad(\c^{-1})\Phi(q) + \c^{-1}\d\c
\end{equation}
for $q\in \sQ$ and $\c\in \Omega GL(n)$, where $\d$ is differentiation in the circle direction.
Denoting by $\Theta$ the (left-invariant) Maurer-Cartan form on $GL(n)$, a Higgs field $\Phi$ and an $\Omega GL(n)$-connection $\sA$ on $\Omega GL(n)$-bundle together determine the form
\begin{equation}
\label{eqn:finiteconnconstruct}
A_{(q,\theta,g)} = \ad(g^{-1}) \big( \sA_q(\theta) + \Phi(q)(\theta) d\theta\big) + \Theta_g
\end{equation}
on $\sQ \x\sone \x GL(n)$.
This form descends to the quotient $Q = (\sQ \x\sone \x GL(n))/\Omega GL(n)$ and determines a framed $GL(n)$-connection.
Conversely, a framed $GL(n)$-connection determines a Higgs field and $\Omega GL(n)$-connection on the inverse caloron transform bundle.
The reader is encouraged to think of the caloron correspondence as a smooth version of the suspension construction in topology that moreover includes the connective structure.

An outline of this paper is as follows.
In Section \ref{S:omega} we develop the theory of $\Omega$ bundles and their caloron transforms following \cite{Sch}.
As part of this discussion, we introduce the appropriate notion of connective data on $\Omega$ bundles.
We show that stable isomorphism classes of $\Omega$ bundles over $M$ under the direct sum operation give a smooth model for the odd $K$-theory of any smooth compact manifold $M$ and compute the odd Chern character.
This gives a very natural and straightfoward bundle-theoretic interpretation of the identity $K^{-1}(M) \cong [M,B\Omega GL] \cong [M, GL]$.

Following this, in Section \ref{S:string} we introduce the \emph{string potential form}, which plays the role of the Chern--Simons form for $\Omega$ bundles (see also \cite{Steve} for a related construction). 
Using this string potential and following the techniques of Simons--Sullivan, in Section \ref{S:model} we construct a differential extension of odd $K$-theory, which we call the $\Omega$ model.
Analogously to the Simons--Sullivan constrcution, the $\Omega$ model is built out of stable isomorphism classes of \emph{structured} $\Omega$ bundles, which are equipped with an equivalence class of connective data defined by exactness of the string potential form.

Whilst bearing a close similarity to that of the Simons--Sulivan model, the construction of the $\Omega$ model is not simply a straightforward generalisation.
In particular, Theorem \ref{theorem:inversesexist} on the existence of inverses to structured $\Omega$ bundles requires substantially more work than in the finite rank case.
Another difficulty peculiar to this setting is that while differential $K$-theory can be characterised uniquely up to unique isomorphism, differential \emph{extensions} of odd $K$-theory are not unique; this is not the case for differential extensions of even $K$-theory.
In the proof of Theorem \ref{theorem:odddiffk} we use the caloron correspondence to construct a canonical isomorphism to (some fixed model of) odd differential $K$-theory that factors through the Simons--Sullivan model.
In this fashion, we establish that the $\Omega$ model is indeed a model for odd differential $K$-theory.

Finally, in Section \ref{TWZ} we use the $\Omega$ model to complete the work of Tradler--Wilson--Zeinalian by proving their conjecture of \cite{TWZ}. In loc.~cit.~the authors construct a differential extension of odd $K$-theory via certain equivalence classes of smooth classifying maps into $GL$ and state that this construction should be odd differential $K$ theory.
After recalling the TWZ construction, we provide an explicit isomorphism to the $\Omega$ model, verifying that it does indeed define odd differential $K$-theory.
The proof ought to be viewed as a refined version of the homotopy equivalence $GL\cong \Omega BGL$.

\addtocontents{toc}{\protect\setcounter{tocdepth}{-2}} 
\subsection*{Terminology and conventions}
Throughout this paper, we let $M$ be a compact finite-dimensional manifold, possibly with corners. All vector bundles are taken to be complex and we denote by $\underline{\CC}^n$ the trivial vector bundle $M\x {\CC}^n \to M$.
Unless stated otherwise, everything is assumed to be taking place in the smooth category; using smooth manifolds, smooth vector bundles, etc.
The circle group $\sone$ is  regarded as the quotient $\RR/2\pi\ZZ$ and is equipped with basepoint $0$.
Integration over the fibre for fibrations of the form $M\x\sone \to M$ is always taken with respect to the canonical orientation on $\sone$ inherited from $\RR$.

\addtocontents{toc}{\protect\setcounter{tocdepth}{2}} 

\renewcommand{\theequation}{\arabic{section}.\arabic{equation}}
\section{$\Omega$ vector bundles}
\label{S:omega}
In order to develop the notions underlying our geometric model for odd differential $K$-theory, we first give a simple characterisation of odd $K$-theory $K^{-1}(M)$ in terms of \lq\lq $\Omega$ vector bundles'' (or simply \lq\lq $\Omega$ bundles") that is completely analogous to the description of even $K$-theory $K^0(M)$ as the group of smooth virtual vector bundles over $M$.
Since the underlying objects are smooth bundles, as we shall see this characterisation allows us to define odd differential $K$-theory by introducing the appropriate notion of connective data on $\Omega$ bundles. 
What follows is a brief account of the theory developed in \cite[Chapter 3]{Sch} and the interested reader is referred there for more detail.

To motivate our construction recall that topologically we have
$$
K^{-1}(M) :=\widetilde K^0\big(\Sigma M^+\big) = \widetilde K^0\left( (M \times S^1 ) / M_0 \right) .
$$
Here $M^+ :=M\sqcup \{\ast\}$ is a pointed space with the adjoined basepoint $\{\ast\}$,  $\Sigma M^+$ is its reduced suspension and $M_0 = M \times \{ 0 \}$. 
The tilde denotes reduced $K$-theory, namely the kernel of the pullback-to-basepoint map or equivalently the restriction to virtual bundles of virtual rank zero. The second equality follows from the obvious homeomorphism $\Sigma M^+ \simeq (M\x\sone)/M_0$.
Of course the quotient  $(M \times S^1 ) / M_0$ is rarely a manifold, so we deal with the question of smoothness by regarding  bundles on $(M \times S^1 ) / M_0$ as bundles on $M \times S^1$ that are trivial on $M_0$. 
More precisely, we have

\begin{definition}
\label{defn:vecframe}
A (smooth) {\em framed vector bundle} of rank $n$ over $M \times S^1$ is a pair consisting of a vector bundle $E \to M \times S^1$ of rank $n$ and an isomorphism 
$E_{|M_0} \simeq M_0 \times \CC^n$ where $M_0 = M \times\{ 0 \}$.
A choice of isomorphism $E_{|M_0} \simeq M_0 \times \CC^n$ is a \emph{framing} of $E$ over $M_0$.
\end{definition}

We are interested in obtaining a characterisation of $K^{-1}(M)$ in terms of smooth bundles on $M$, not on $M \times S^1$.
We can pass from one to the other by analogy with the push down or direct image operation of algebraic
geometry.  
If $E \to M \times S^1$ is a vector bundle of rank $n$ and $m \in M$ we define 
\begin{equation}
\label{eqn:1stexample}
\sE_m := \Gamma( \{m \} \times S^1, E)
\end{equation}
and let $\sE$ be the disjoint union of all the $\sE_m$ with the obvious projection map to $M$.  Notice that if $U \subseteq M$ is open 
and $s \in \Gamma(U \times S^1, E)$, then $\check s $  defined by $\check s(m)(\theta) = s(m, \theta)$ 
is a section of $\sE$ over $U$. 
Smoothness of $\sE$ is determined by requiring that these be exactly all the smooth sections of $\sE$ over $U$. 

To be able to characterise precisely which kinds of bundles arise in this way we need to explore the structure 
of the fibres.  They all have the form $\Gamma(S^1, F)$ for some complex vector bundle $F \to S^1$ which we could, of course, take to be trivial. 
We note that they are all acted on by the ring $L\CC := C^\infty(S^1, \CC)$ and  our first observation is essentially the Serre-Swan theorem
\begin{proposition}
\label{theorem:module}
An $L\CC$-module $N$ is isomorphic to the space of sections $\Gamma(S^1, F)$ of some vector bundle $F \to S^1$ if and only if $N$ is finitely generated and free.
\end{proposition}

The following gives a characterisation of the automorphisms of the fibres of bundles constructed via \eqref{eqn:1stexample}.

\begin{proposition} 
\label{prop:aut-mod}
A $\CC$-linear automorphism $f \colon L\CC^n \to L\CC^n$ is an $L\CC$-module automorphism if and only 
if it arises as pointwise multiplication by an element of $L GL(n) := C^\infty(S^1, GL(n))$.
\end{proposition}
\begin{proof}
Consider the $L\CC$-module $L\CC^n$. Evaluation at $\theta \in S^1$ gives the surjection $\ev_\theta \colon L\CC^n \to \CC^n$ with kernel the ideal $ I_\theta L\CC^n $ of $n$-tuples of functions $\sone\to \CC$ vanishing at $\theta$.  It follows that we have the induced isomorphism $\ev_\theta \colon L\CC^n /  I_\theta L\CC^n  \simeq \CC^n$ and if $f \colon L\CC^n \to L\CC^n$ is an $L\CC$-module isomorphism then it clearly preserves $I_\theta L\CC^n$ and induces a linear map $f(\theta) \colon \CC^n \to \CC^n$. 
By definition, for any $v \in L\CC^n$ we then have $(fv)(\theta) = f(\theta) v(\theta)$, so $f$ is completely determined by the $f(\theta)$.  
It follows that there is a map $f^{-1}$ with 
$f^{-1}(\theta) f(\theta) = 1$,  \mbox{i.e.} ~$f(\theta)$ is invertible.  
Finally if we apply $f$ to the constant map $\theta\mapsto \mathbf{e}^i$, the $i$-th standard basis vector of $\CC^n$, we obtain the $i$-th column of $f(\theta)$, which is an element of $L \CC^n$ and hence smooth.
So the assignment $\theta \mapsto f(\theta)\in GL(n)$ is smooth.  
The converse is straightforward. 
\end{proof}

So far we have characterised the bundles $\sE$ as Fr\'echet vector bundles over $M$ whose fibres are free finitely generated $L\CC$-modules.
Alternatively, they may be viewed as Fr\'echet vector bundles whose fibres are isomorphic to $L\CC^n$ and whose frame bundle  has a reduction to $LGL(n)$.  
In the case that $\sE$ is constructed from a framed vector bundle $E \to M\x\sone$, the canonical trivialisation of the fibre of $E$ over $(m,0)$ yields an extra constraint on $\sE$.
To determine this condition, let $N$ be a free finitely generated module over $L\CC$ so that there is an $L\CC$-module isomorphism $N \cong L\CC^k$, with $k$ the rank of $N$.
Let $I_\theta N$ be the kernel of the composition of such an isomorphism with $\ev_\theta$.
Motivated by  the proof of Proposition \ref{prop:aut-mod},  for $n \in N$ we say that 
$n$ {\em vanishes at $\theta \in S^1$} if $n \in I_\theta N$ and similarly that two elements $n_1$ and $n_2$ {\em agree at $\theta$} if $n_1 - n_2$ vanishes at $\theta$. 
If $ g \in \Aut(N)$ then we say that $g$ is {\em the identity at $\theta$} if for every $n \in N$ we have that $g(n) $ and $n$ agree at $\theta$. 
If  $N = L\CC^n$ as in Proposition \ref{prop:aut-mod} then such a $g$ corresponds to an element of $LGL(n)$ that is the identity at $\theta$.
 This implies the

\begin{lemma}
There is a canonical identification of the subgroup of based loops $\Omega GL(n) \subset LGL(n)$ with the subgroup $\Aut_0(L\CC^n)$ of $L\CC$-module automorphisms of $L\CC^n$ which are the identity at $0$. 
\end{lemma}

If $N$ is a free finitely generated module of rank $n$ and $\theta \in S^1$ we define 
$$
N_\theta := \frac{N}{I_\theta N}
$$
which is a vector space of dimension $n$.
We say that $N$ is {\em framed at $0$} if we have chosen an isomorphism between $N_0$ and $\CC^n$.  In the case of $L\CC^n$ we take the canonical such framing.  
We define $\Iso_0(L\CC^n, N)$ to be all $L\CC$-module isomorphisms that preserve the framing at $0$, then we have 

\begin{proposition}
Let $N$ be a free finitely generated projective $L\CC$-module of rank $n$ framed at $0$. 
Then $\Iso_0(L\CC^n, N)$ is a right $\Aut_0(L\CC^n)$- (or $\Omega GL(n)$)-torsor.
\end{proposition}
Recall that a set $X$ being a right $G$-torsor means that $G$ has a free and transitive right action on $X$.

For convenience let us call a free finitely generated $L\CC$-module of rank $n$ an $\Omega$-module of rank $n$.  We are
interested in bundles of $\Omega$-modules. If $\sE \to M$ is a fibration whose fibres are $\Omega$-modules we say 
it is locally $\Omega$-trivial if we  can cover $M$ with  open sets $U \subset M$ and choose isomorphisms  $ U \times L\CC^n \to E_{|U} $ which are fibrewise $\Omega$-module isomorphisms. We also call the latter \emph{$\Omega$ frames}.

\begin{definition}
\label{defn:ombun}
An \emph{$\Omega$ vector bundle over $M$} is a locally $\Omega$-trivial fibration $\sE \to M$ whose fibres are $\Omega$-modules.
\end{definition}
The \emph{rank} of an $\Omega$ bundle $\sE$, denoted $\rk\sE$, is the locally constant function defined as the rank of any fibre $\sE_m$ as an $L\CC$-module. A \emph{morphism} of $\Omega$ bundles is a smooth map $\sE \to \sF$ of Fr\'{e}chet vector bundles that restricts to a map of $\Omega$-modules on each fibre.

Notice that if   $\sE $ is constructed from some framed vector bundle $E \to M \times S^1$ as in \eqref{eqn:1stexample} then local $\Omega$-triviality can be seen by trivialising $E $ over $U \times S^1$.

\begin{example}
The \emph{trivial $\Omega$ bundle of rank $n$ over $M$} is simply
\[
\underline{L\CC}^n := M\x L\CC^n \lo M.
\] 
\end{example}

\begin{example}
\label{example:universalomegabundles}
For any positive integer $n$, recall (cf. \cite[Section 3.1.2]{V}) that the path fibration $\pi\colon PGL(n) \to GL(n)$ is a model for the universal $\Omega GL(n)$-bundle.
The total space $PGL(n)$ is the Fr\'{e}chet manifold of paths $p\colon \RR \to GL(n)$ such that $p(0) = \id$ with the additional property that $p^{-1}\d p$ is $2\pi$-periodic.
The projection map is simply evaluation at $2\pi$.

Via the standard representation of $\Omega GL(n)$ on $L\CC^n$, we may view $PGL(n)$ as the bundle of $\Omega$ frames for the associated bundle
\[
\sE(n) := PGL(n) \x_{\Omega GL(n)} L\CC^n \lo GL(n).
\]
Note this implies $\sE(n) $ is a classifying bundle for $\Omega$ vector bundles of rank $n$.
\end{example}

Following \cite{MV}, we call  $\sE \to M$ as constructed in \eqref{eqn:1stexample}  the {\em inverse caloron transform} of the framed bundle $E \to M \times S^1$. The reason for this is the following proposition which we leave to the interested reader.

\begin{proposition}
Let $\cF(E) \to M \times S^1$ be the $GL(n)$ frame bundle of the framed bundle $E \to M 
\times S^1$. 
Then the vector bundle associated to the inverse caloron transform of $\cF(E)$ by the natural 
representation of $\Omega GL(n)$ on $L\CC^n$ is naturally isomorphic to $\sE$. 
\end{proposition}

Given any $\Omega$ bundle $\sE\to M$ we can 
define, as above, for any $m \in M$ and $\theta \in S^1$ an $n$-dimensional complex vector space
$$
E_{(m, \theta)} := \sE_m / I_\theta \sE_m .
$$
We claim that 
$$
E := \bigcup_{(m, \theta)  \in M \times S^1} E_{(m, \theta)} \to M \times S^1 
$$
is naturally a rank $n$ vector bundle.  The simplest way to see this is to notice that a local fibrewise $L\CC$-module isomorphism $\xi\colon U \times L\CC^n \to \sE_{|U}$ induces a vector bundle isomorphism $\check \xi \colon 
\underline \CC^n \to \cup_{(m,\theta) \in U\x\sone} E_{(m , \theta)}$. Moreover 
if $\eta $ is another  $L\CC$-module isomorphism with $\xi = \eta g $ for some $g \in \Omega GL(n)$, then $\check \xi = \check \eta g(\theta)$. 
It follows that if we choose local trivialisations of $\sE$ relative to an open cover $\{ U_\a \}_{\a \in I}$ that are pointwise $\Omega$ frames, then they induce local trivialisations of $E$. Moreover if $g_{\a\b} \colon U_\a \cap U_\b \to \Aut_0(L\CC^n) = \Omega GL(n)$ are local trivialisations of $\sE$ then 
$\check g_{\a\b} \colon (U_\a \times S^1) \cap (U_\b \times S^1) \to GL(n)$ are local trivialisations of $E$.   This discussion also shows that above the 
point $\theta = 0 $ we can use the fact that $g(0) = 1$ to show that $E$ is framed over $M_0$.  

Given an $\Omega$ bundle $\sE \to M $ we call the framed vector bundle $E \to M \times S^1$ constructed in this way the {\em caloron transform} of $\sE$.
Observe that the rank of $\sE$ is exactly the rank of its caloron transform $E$. 
 We leave it again to the interested reader to prove the following proposition.

\begin{proposition} 
Let  $\cF(\sE) \to M$ be the $\Omega GL(n)$ frame bundle of the $\Omega$ bundle $\sE \to M$. The $\CC^n$-vector
bundle associated to the caloron transform of $\cF(\sE)$ is naturally isomorphic to $E$.
\end{proposition}

We know from \cite[Proposition 3.2]{MV} that the caloron and inverse caloron transforms for principal bundles give an equivalence of categories and thus the same will be true for associated vector bundles.  Alternatively this may be seen directly using the following. In the case of the module $L\CC^n$ we can think of the evaluation map as the projection 
$$
\ev_\theta \colon L\CC^n \to L\CC^n / I_\theta L\CC^n  \simeq \CC^n .
$$
Similarly, for any $m \in M$,  we can define the {\em evaluation at $\theta$ map}
$$
\ev_\theta \colon \sE_m \to \sE_m / I_\theta \sE_m   = E_{(m, \theta)}
$$
which induces an isomorphism
$$
\ev \colon \sE_m \to \Gamma( \{m\} \times S^1,  E )
$$
and hence an  isomorphism 
\begin{equation}
\label{eqn:1stmoduleiden}
\overline{\ev} \colon \Gamma(M,\sE) \simto \Gamma(M\x\sone,E)
\end{equation}
of modules over $C^\infty(M\x\sone,\CC)$\footnote{the action of $f\in C^\infty(M\x\sone,\CC)$ on $s \in \Gamma(M,\sE)$ is given by $(f\cdot s)(m) := \check f(m) s(m)$ where $\check f(m) \colon \theta \mapsto f(m,\theta)$ is a loop in $\CC$.} that preserves the framing at $0\in \sone$. 
In particular we have the following

\begin{proposition} 
\label{prop:1cal}
If $\sE \to M $ is an $\Omega$ bundle, then it is naturally isomorphic to the inverse caloron transform of its caloron transform $E \to M \times S^1$.
\end{proposition}

Conversely given a framed vector bundle $E \to M\x\sone$ and writing $\sE \to M$ for its inverse caloron transform, there is an obvious isomorphism
\begin{equation}
\label{eqn:2ndmoduleiden}
\underline{\ev} \colon \Gamma(M\x\sone,E) \simto \Gamma(M,\sE)
\end{equation}
 of $C^\infty(M\x\sone,\CC)$-modules that preserves the framing at $0\in \sone$.
\begin{proposition}
\label{prop:2cal}
If $E \to M\x\sone$ is a framed vector bundle, then it is naturally isomorphic to the caloron transform of its  inverse caloron transform $\sE\to M$.
\end{proposition}

Notice that the caloron transform and the inverse caloron transform only invert each other up to natural isomorphism---they are \emph{pseudo-inverses}---and are best viewed as giving an equivalence of categories.
This is the point of view of \cite[Section 3.1.1]{Sch}.

\subsection{Higgs fields}
By definition, the fibres of any $\Omega$ bundle of rank $n$ are $L\CC$-modules arising as the space of sections $\Gamma(S^1,F)$ for some $\CC^n$-bundle $F \to S^1$. 
As such we can consider connections on $F$;
 usually they give rise to a covariant derivative 
$\nabla \colon \Gamma(S^1,F) \to \Gamma(S^1, T^*S^1 \otimes F)$
but  using the canonical trivialisation of  $T^*\sone$ we may regard the connection as a map $\nabla \colon \Gamma(S^1,F) \to \Gamma(S^1,  F)$.
Thus a connection on $F$ is equivalent to an $L\CC$-derivation $\delta \colon \Gamma(S^1,F) \to \Gamma(S^1,F)$.

We define an $L\CC$-derivation of the $L\CC$-module $N$ is a linear  map $\delta \colon N \to N$ satisfying 
$\delta(s n) = (\partial_\theta s) n + s \delta (n)$ for all $s \in L\CC$ and $n\in N$, with $\partial_\theta$ partial differentiation with respect to $\theta$.

\begin{example}
Every derivation on $L \CC^n$ is of the form $ \partial_\theta + a $ for $a \in L\mathfrak{gl}(n)$.
\end{example}

If $N$ is an $L\CC$-module, denote the collection 
of all such derivations by $\Der(N)$ and note that an affine combination of 
derivations is again a derivation.  If $\sE \to M$ is an $\Omega$ bundle we can apply $\Der$ fibrewise to 
obtain a bundle $\Der(\sE) \to M$. Then we have 

\begin{definition} A {\em Higgs field} $\phi$ for an $\Omega$ vector bundle $\sE \to M$ is a smooth section 
of $\Der(\sE) \to M$. 
\end{definition}

\begin{example}
The \emph{trivial Higgs field on $\underline{L\CC}^n$} is simply the fibrewise $L\CC$-module derivation $\d \colon v\mapsto \d_\theta v$ that differentiates along the circle.
\end{example}

The existence of Higgs fields is guaranteed by a standard partition of unity argument. Alternatively, we may proceed as follows.
A connection on a framed vector bundle  $E \to M \times S^1$ is \emph{framed} if it restricts to the trivial connection on $E_ {|M_0}$.  
Let $\sE$ be the inverse caloron transform of the framed vector bundle $E\to M\x\sone$ so that $\sE_m$ is the space of sections of $E$ over $\{m \} \times S^1$.

\begin{proposition} 
The restriction of a framed connection on the framed bundle $E \to M \times S^1$ to the circle direction 
defines a Higgs field on its inverse caloron transform $\sE$. 
Moreover, since any $\Omega$ bundle is isomorphic to the inverse caloron transform of some framed bundle over $M\x\sone$, Higgs fields exist.
\end{proposition}

Explicitly, let $\d_\theta$ denote the canonical vector field on $E$ in the circle direction.
Let $\nabla$ be a framed connection on the framed bundle $E$.
If $\sE$ is the caloron transform of $E$, then for any $v \in \sE_m$ we choose a section $s \in \Gamma(\sE)$ extending $v$.
The Higgs field determined by $\nabla$ is then
\begin{equation}
\label{eqn:definingHF}
\phi(v) = \underline\ev\, \Big(\!\nabla_{\d_\theta}\big(\overline{\ev}(s) \big)\!\Big)(m)
\end{equation}
where $\overline{\ev}$, $\underline{\ev}$ are from \eqref{eqn:1stmoduleiden} and \eqref{eqn:2ndmoduleiden} and we have omitted reference to the natural isomorphisms coming from the caloron correspondence.
In point of fact, we do not require the connection $\nabla$ to be framed in the above result---this condition is required later in the discussion on module connections.

We remark that there is an obvious pullback operation on Higgs fields. Furthermore, by expressing the connection acting on a local basis in terms of its connection 1-forms in the standard fashion, we have
\begin{lemma}
Higgs fields on the $\Omega$ bundle $\sE$ are in bijective correspondence with Higgs fields on its $\Omega GL(n)$ frame bundle $\cF(\sE)$.
\end{lemma}

\begin{example}
\label{example:universalHF}
Recalling Example \ref{example:universalomegabundles}, we note that the map $\Phi \colon p\mapsto p^{-1}\d p$ satisfies \eqref{eqn:twistedeq} so is a Higgs field on $PGL(n)$.
Applying the above Lemma gives a Higgs field $\phi(n)$ on the associated $\Omega$ bundle  $\sE(n) \to GL(n)$. 
\end{example}

An important notion is that of \emph{holonomy} of a Higgs field.
Recall from \cite[Proposition 4.3]{MV} that if $\Phi$ is a Higgs field on the principal $\Omega G$-bundle $\sQ \to M$, then the \emph{holonomy} of $\Phi$ is given by solving
\[
\Phi(q) = g^{-1} \d g
\]
for the path $g = g(q)$ subject to the initial condition $g(0) = 1$.
This gives a smooth $\Omega G$-equivariant map $\hol_\Phi \colon \sQ \to PG$ that descends to a classifying map $\hol_\Phi \colon M \to G$ for the bundle $\sQ$.

 In the case of the $\Omega$ bundle $\sE \to M$ equipped with Higgs field $\phi$, we define the \emph{holonomy} of $\phi$ as follows.
Up to natural isomorphism, $\sE$ may be considered as the inverse caloron transform of some framed bundle $E \to M\x\sone$, so that $\sE_m = \Gamma(\{m\}\x\sone,E)$, and we set $n=\rk\sE$.
Moreover, by Proposition \ref{prop:geometriccal} below, under this isomorphism the Higgs field is given fibrewise by the expression \eqref{eqn:definingHF} for some framed connection $\nabla$ on $E$.
The holonomy of $\phi$ is then the map
\[
\hol_\phi \colon M \lo GL(n)
\]
that sends $m$ to the holonomy of $\nabla$ around the loop $\theta\mapsto (m,\theta)$ starting at $(m,0)$.
It is not too hard to see that if $\Phi$ is the corresponding Higgs field on the $\Omega GL(n)$ frame bundle $\cF(\sE)$, then $\hol_\phi = \hol_\Phi$.
It follows from this that the pullback of $\sE(n)$ by $\hol_\phi$ is isomorphic to $\sE$ so that $\hol_\phi$ is a classifying map for $\sE$.

\subsection{Module connections}
In addition to Higgs fields, there is another important notion of connective data on $\Omega$ bundles.
First consider an ordinary connection $\triangle \colon \Gamma(M,\sE) \to \Omega^1(M, \sE)$, thought of as a vector bundle covariant derivative. 
The fibre
$T_m^*M \otimes \sE_m$ may be viewed as an $L\CC$-module in an obvious way, so it is natural to ask if the connection is a derivation for the $L\CC$-module structure; if it is then it must also preserve sections that vanish at $0\in\sone$.
That is, such a connection preserves sections of the form $s \colon m \mapsto s(m) \in I_0 E_m$ so it induces an action on sections of the bundle $\sE_0$ whose fibre over $m$ is $\sE_m/I_0\sE_m$. 

\begin{definition}
A connection $\triangle \colon \Gamma(M,\sE) \to \Omega^1(M, \sE)$ is a  {\em module connection}
if it is a derivation for the  $L\CC$-module structure on the fibres and if it induces the trivial connection $\Gamma(M, \sE_0) \to \Omega^1(M, \sE_0)$.
\end{definition}

\begin{example}
The \emph{trivial module connection on $\underline{L\CC}^n$} is simply the ordinary trivial connection  $\de \colon v \mapsto d  v$ considered as an $L\CC$-module derivation.
\end{example}

Module connections always exists on $\Omega$ bundles.
This may be seen directly by recalling that every $\Omega$ bundle $\sE \to M$ is naturally isomorphic to the inverse caloron transform of some framed bundle $E \to M\x\sone$.
Picking a framed connection $\nabla$ on $E$, by using the canonical trivialisation of $T^\ast\sone$ the expression
\[
\triangle := \underline{\ev} \circ \nabla \circ \overline{\ev}
\]
may be viewed as a map $\Gamma(M,\sE) \to \Omega^1(M,\sE)$, once again omitting the natural isomorphism arising from the caloron correspondence.
Since $\nabla$ is a connection, it follows that $\triangle$ is a derivation for the $L\CC$-module structure on the fibres of $\sE$.
Moreover, the framing condition on $\nabla$ implies that the induced connection $\Gamma(M,\sE_0) \to \Omega^1(M,\sE_0)$ is trivial so that $\triangle$ is a module connection---we call this the module connection determined by $\nabla$ via the inverse caloron transform.

It is not difficult to check that the pullback of a module connection is once again a module connection.

\begin{lemma}
Module connections on the $\Omega$ bundle $\sE$ are in bijective correspondence with $\Omega GL(n)$-connections on its associated $\Omega GL(n)$ frame bundle
$\cF(\sE)$. 
\end{lemma}

\begin{example}
\label{example:universalmodconn}
There is a well-known class of connections on the path fibration $PGL(n) \to GL(n)$, given as follows (see also \cite[Example 3.5]{MV}).
Let $\Theta$ denote the (left-invariant) Maurer-Cartan form for the pointwise multiplication on $PGL(n)$ and let $\widehat\Theta$ denote the \emph{right-invariant} Maurer-Cartan form on $GL(n)$.
Choose a smooth function $\alpha \colon \RR\to\RR$ such that $\alpha(t) = 0$ for $t\leq0$ and $\alpha(t) =1$ for $t \geq 1$, then the expression 
\[
\sA_\alpha = \Theta - \alpha \ad(p^{-1}) \pi^\ast\widehat\Theta
\]
defines an $\Omega GL(n)$-connection at $p \in PGL(n)$.
By the above lemma, $\sA_\a$ determines a module connection $\t_\alpha(n)$ on the $\Omega$ bundle  $\sE(n) \to GL(n)$.  
Of course, this module connection depends on the choice of smooth function $\alpha$.
\end{example}

When we are given a framed vector bundle equipped with framed connection, we have already seen how to define a module connection and a Higgs field on its inverse caloron transform.
Conversely, if $\sE \to M$ is a given $\Omega$ bundle equipped with module connection $\t$ and Higgs field $\phi$, we obtain a framed connection on the caloron transform $E \to M\x\sone$ as follows.
First observe that, via a natural isomorphism arising from the caloron correspondence, we may identify the expressions
\[
\overline{\ev} \circ \triangle\circ \underline{\ev}
\;\;
\mbox{ and }
\;\;
\overline{\ev} \circ \phi\circ \underline{\ev}
\]
respectively with the $M$ and $\sone$ components of a connection on $E$.
Adding these together thus gives a connection on $E$, which can be shown to be framed---this is the framed connection determined by $\t$ and $\phi$.

\begin{proposition}
\label{prop:geometriccal}
The natural isomorphisms of Propositions \ref{prop:1cal} and \ref{prop:2cal} preserve the connective data.
\end{proposition}

\begin{remark}
As with the topological caloron correspondence above, the caloron correspondence for bundles with connective data may be phrased as an equivalence of categories.
This formulation, developed in \cite[Section 3.2.1]{Sch}, shows that any map of $\Omega$ bundles $\sE \to \sF$ preserving the module connection and Higgs field gives rise to a map $E \to F$ of the caloron transforms that preserves the associated framed connections; the converse also holds.
\end{remark}

\subsection{Operations}

There are two operations of interest on $\Omega$ bundles.  The first is the Whitney sum of vector bundles. Notice that if $F_1, F_2$ 
are vector bundles over $S^1$ then there is a natural isomorphism
$$
\Gamma(S^1, F_1) \oplus \Gamma(S^1, F_2) \cong \Gamma(S^1, F_1 \oplus F_2)
$$
of $L\CC$-modules.  It follows that the inverse caloron transform of $E_1 \oplus E_2$ is $\sE_1 \oplus \sE_2$ and vice versa.

Similarly applying the tensor product of $L\CC$-modules we have
$$
\Gamma(S^1, F_1) \otimes_{L\CC} \Gamma(S^1, F_2) \cong \Gamma(S^1, F_1 \otimes F_2).
$$
So for $L\CC$-modules $N, N'$ we write $N\oast N' := N\otimes_{L\CC} N'$ and we define the {\em honed tensor product} of $\Omega$ bundles $\sE_1$ and $\sE_2$, to distinguish it from the vector space tensor product, by
$$
(\sE_1 \oast \sE_2)_m := (\sE_1)_m \otimes_{L\CC} (\sE_2)_m.
$$
It follows that the inverse caloron transform of $E_1 \otimes E_2$ is $\sE_1 \oast \sE_2$ and vice versa.

Finally we note that if $F \to S^1$ is a vector bundle then 
$$
\Hom_{L\CC}( \Gamma(S^1, F), L\CC) \cong \Gamma(S^1, F^*)
$$
and we define for any $L\CC$-module $N$ the dual module $N^\ast := \Hom_{L\CC}(N, L\CC)$. Similarly for any $\Omega$ bundle $\sE\to M$ we define the $\Omega$ bundle $\sE^\ast\to M$ fibrewise. If $\sE$ is the inverse caloron transform of the framed bundle $E$, then $\sE^\ast$ is the inverse caloron transform of the dual $E^*$ equipped with the dual framing, and vice versa.  Likewise if $N$ is an $L\CC$-module we define $\overline N$ to be the conjugate module where the $L\CC$-action is anti-linear,  \emph{i.e.} ~$sn  = \overline s n$ for $s \in L\CC$ and $n \in N$.  Again we can show that the inverse caloron transform of $\overline E$ is $\overline \sE$ and vice-versa.

\subsection{Hermitian structures}

For the discussion in Section \ref{TWZ} also need to consider $\Omega$ bundles with structure group $\Omega U(n)$. 
Let $F \to S^1$ be a Hermitian vector bundle, with Hermitian structure $\langle\ ,\ \rangle$. For $f, g \in \Gamma(S^1, F)$
we then have $\left(z \mapsto \langle f(z), g(z) \rangle \right)\in L\CC$.  Motivated by this we have the following

\begin{definition}
Let $N$ be a free finitely generated $L\CC$-module. A \emph{Hermitian structure} on $N$  is an $L\CC$-module map 
$$
\llan\ ,\ \rran \colon N \oast \overline N \to L\CC.
$$
such that
\begin{enumerate}
\item
$\llan v, {v} \rran (\theta) > 0$ whenever $v\notin I_\theta N$,
\item
$\llan v,{w} \rran(\theta) = \overline{\llan w,{v} \rran}(\theta)$ for all $v \in N$, $w \in\overline N$ and $\theta \in \sone$.
\end{enumerate}
\end{definition}

Notice that if $v\in I_\theta N$ it has the form $v = s w$ for some $s \in L\CC$ vanishing at $\theta$, so that since $\llan v, x \rran (\theta) =  s(\theta) 
\llan w, x \rran (\theta ) = 0$ for any $x\in N$ there is an induced Hermitian inner product on the finite dimensional vector space $N_\theta = N/I_\theta N$. 
If $N$ is a free finitely generated module framed at zero, then  the Hermitian structure is \emph{framed at zero} if $(v, w) \mapsto \llan v, w \rran  (0) $ induces the standard inner product on $\CC^n$. 

\begin{definition}
A \emph{Hermitian $\Omega$  bundle} is an $\Omega$ vector bundle with a smooth fibrewise choice of a Hermitian structure which is framed at zero. 
\end{definition}

Now it is not difficult to prove the following results.

\begin{proposition}
If $\sE$ is a Hermitian $\Omega$ vector bundle its frame bundle $\cF(\sE)$ has a natural reduction to $\Omega U(n)$. 
\end{proposition}

\begin{proposition} 
If $\sE$ is a Hermitian $\Omega$ vector bundle its caloron transform $E$ is a framed Hermitian vector bundle and conversely, if $E $ is a framed
Hermitian vector bundle its inverse caloron transform is a Hermitian $\Omega$ vector bundle.
\end{proposition}

In the above proposition, a \emph{framed} Hermitian vector bundle $E \to M\x\sone$ is a framed bundle equipped with a Hermitian structure such that the sections over $M_0$ determined by the framing are orthonormal.

Given a Hermitian $\Omega$ bundle $\sE\to M$ we can also require that module connections and Higgs fields on $\sE$ be compatible with the Hermitian structure.
A module connection $\t$ is \emph{compatible with (the Hermitian structure) $\llan\ ,\ \rran$} if
\[
d\llan v,w \rran = \llan\t(v),w \rran+  \llan v,\t(w) \rran  
\] 
for all $v,w \in \Gamma(M,\sE)$.
Similarly a Higgs field $\phi$ is \emph{compatible with $\llan\ ,\ \rran $} if
\[
\d_\theta\llan v,w \rran = \llan\phi(v),w \rran+  \llan v,\phi(w) \rran  
\] 
for all $v,w \in \Gamma(M,\sE)$.
Compatible module connections and Higgs fields on $\sE$ are in bijective correspondence with $\Omega U(n)$-connections and Higgs fields on the $\Omega U(n)$ frame bundle of $\sE$.

\subsection{Odd $K$-theory} 

We conclude this section by elucidating the role of $\Omega$ bundles in $K$-theory. If $X$ is a topological space let  $\vect(X)$ denote the semigroup of isomorphism classes of complex vector bundles over $X$ under direct sum. Recall that $K^{-1}(M)$ is defined by the Grothendieck group completion of $\vect\big((M\x\sone)/M_0\big)$, followed by restriction to elements of virtual rank zero. Note that we are regarding $(M\x\sone)/M_0$ as a topological space. We shall construct a natural semigroup isomorphism
\[
\cR\colon \ovect(M) \lo \vect\big((M\x\sone)/M_0\big) ,
\]
where $\ovect(M)$ is the semigroup of isomorphism classes of $\Omega$ vector bundles over $M$ under direct sum.

The map $\cR$ is constructed as follows.
For any isomorphism class $[\sE] \in \ovect(M)$ pick a representative $\sE$ and let $E\to M\x\sone$ denote its caloron transform.
Note that the framing of $E$ over $M_0$ determines a trivialisation $s \colon E_{|M_0} \to M_0\x\CC^n$. 
We now define an equivalence relation on $E_{|M_0}$ by setting
\[
e \sim e' \ \text{if and only if}\ \pr_2\circ \,s(e) = \pr_2\circ \,s(e') ,
\]
where $\pr_2 \colon M_0\x\CC^n\to \CC^n$ is the projection onto the second factor, and 
extend it to all of $E$ by the identity. The set of equivalence classes, denoted $E/s$, is then a vector bundle over $(M\x\sone) /M_0$ \cite[Lemma 1.4.7]{AtiyahK}. Defining the map $\cR$ by $\sE \mapsto E/s$, it is straightforward to check that it is well defined on isomorphism classes. Moreover, it is manifestly natural and extends to a semigroup homomorphism by the properties of the caloron transform since $E/s\oplus F/r\cong (E\oplus F)/(s\oplus r)$, where $s\oplus r$ is the obvious direct sum trivialisation.
\begin{proposition}
$\cR$ is an isomorphism of semigroups.
\end{proposition}
\begin{proof}
For surjectivity, take any complex vector bundle $E \to (M\x\sone)/M_0$ and let $\pi\colon M\x\sone \to (M\x\sone)/M_0$ denote the quotient map.
A choice of framing for $E$ over the point $M_0/M_0$ induces a framing for the lift $\pi^\ast E$ over $M_0$. Thus we have a trivialisation $s\colon (\pi^\ast E)_{|M_0} \to M_0\x\CC^n$ and $(\pi^\ast E)/s$ is clearly isomorphic to $E$. Moreover, we can always find a smooth framed vector bundle $F \to M\x\sone$ that is isomorphic  to $\pi^\ast E$ and hence isomorphic to $E$ after passing to the quotient. But $F$ lies in the image of the caloron transform, so we conclude that $\cR$ is surjective on isomorphism classes. 

To establish injectivity, suppose that $\cR(\sE) = \cR(\sF)$ where $\cR(\sE) = E/s$ and $\cR(\sF) = F/r$.
Recall the Stiefel bundles $V_m(\CC^k) \to \Gr_m(\CC^k)$ where
\[
\Gr_m(\CC^k) := \{ W \subset \CC^k \mid W \mbox{ is a subspace and $\dim W = m$}\} 
\]
is the Grassmannian and
\[
V_m(\CC^k) := \{ (v,W) \mid v \in W\mbox{ and } W \in \Gr_m(\CC^k) \}\subset \CC^k \x\Gr_m(\CC^k).
\]
We emphasise that $V_m(\CC^k) \to \Gr_m(\CC^k)$ has a canonical framing over the point $x_0:= \mathrm{span}\{\mathbf{e}_1,\dotsc,\mathbf{e}_m\} \in \Gr_m(\CC^k)$ corresponding to the trivialisation
\[
\left(\sum_{i=1}^mf^i\mathbf{e}_i,x_0\right) \longmapsto \left(x_0,\sum_{i=1}^m f^i\mathbf{e}_i\right) \in \{x_0\} \x\CC^m,
\] where $\mathbf{e}_i$ denotes the $i$-th standard basis vector in $\CC^m$.

By \cite[Corollary 3.1.24]{Sch}, letting $m=\rank E= \rank F$ there is some $k$  such that $E\cong f^\ast V_m(\CC^k)$ and $F\cong g^\ast V_m(\CC^k)$ for some smooth maps $f,g\colon M\x\sone \to \Gr_m(\CC^k)$ sending $M_0$ to $x_0$.
We may find a smooth homotopy from $f$ to $g$ that is constant on $M_0$, thus $f^\ast V_m(\CC^k)$ and $g^\ast V_m(\CC^k)$ are smoothly isomorphic as bundles with framing.
This gives that $E$ and $F$ are isomorphic as smooth bundles with framing and so injectivity of $\cR$ follows from the caloron correspondence.
\end{proof}

Consider the Grothendieck group completion $K(\ovect(M)) $ of the semigroup $\ovect(M)$ and let 
\[
\rk\colon K(\ovect(M)) \lo \ZZ
\]
be the homomorphism sending $[\sE]-[\sF] \mapsto \rk\sE-\rk\sF$. 
Then since $\cR$ is an isomorphism such that $\rk\sE =   \rank \cR(\sE)$ we have
\begin{theorem}
The odd $K$-theory of $M$ is isomorphic to $\cK^{-1}(M) := \ker\rk$.
\end{theorem}

In other words, elements of $\cK^{-1}(M)$ are virtual $\Omega$ bundles of virtual rank zero.
By the caloron correspondence and the Serre-Swan theorem, every element of $\cK^{-1}(M)$ may be written in the form $\sE - \ul{L\CC}^n$ where $n = \rk \sE$.

Another way of showing that $\cK^{-1}(M)$ is the odd $K$-theory of $M$ is to use the homotopy theoretic model for $K^{-1}$ as follows.
 Let $GL$ denote the stabilised general linear group and recall that $K^{-1}(M) \cong [M,GL]$. The group operation on $[M,GL]$ is given by $[g]+ [h] = [g\oplus h]$, where $g\oplus h$ is the pointwise block sum of $g$ and $h$.
At the level of homotopy, the group operations given by block sum and the matrix product are equal so the inverse of $[g] \in K^{-1}(M)$ is the homotopy class of the map $g^{-1}\colon x\mapsto (g(x))^{-1}$. Define a homomorphism $\cK^{-1}(M) \to [M,GL]$ by sending the virtual $\Omega$ vector bundle $\sE - \underline{L\CC}^n$ to the homotopy class $[g]$, where $g$ is any smooth classifying map (equivalently, Higgs field holonomy) for $\sE$. It is not difficult to verify that this is in fact an isomorphism of groups.

This classifying map approach underlies the model of Tradler--Wilson--Zeinalian \cite{TWZ}, which constructs a differential extension of odd $K$-theory via equivalence classes of maps into the smooth classifying space. 
In section \ref{TWZ} we show that their model defines odd differential $K$-theory by giving an isomorphism to the $\Omega$ model of section \ref{S:model}.
The above paragraph ought to be viewed in this context as the topological precursor to the proof of section \ref{TWZ}.

\section{The string form and the string potential}
\label{S:string}
In this section we introduce the string form and its transgression form using module connections and Higgs fields. These give a characterisation of the odd Chern character and its Chern--Simons form for $\Omega$ bundles.

\subsection{The string form}
Recall that for a complex vector bundle $E\to M$ equipped with connection $\nabla$ with curvature $R$, the Chern character is a closed even complex-valued form defined by
\begin{equation}
\label{eqn:cherncharacterform}
\Ch(\nabla) = \sum_{j=0}^\infty \frac{1}{j!}\bigg(\frac{1}{2\pi i}\bigg)^j \tr\big( \underbrace{R\wedge\dotsb \wedge R}_{\text{$j$ }}\big) .
\end{equation}
It is a standard fact that the cohomology class of $\Ch(\nabla)$ is independent of the choice of connection and that the even Chern character map
\[
ch\colon K^0(M) \lo H^{even}(M;\CC)
\]
sending $E-F \mapsto [\Ch(\nabla)-\Ch(\tilde \nabla)]$ is a ring homomorphism that is an isomorphism after tensoring with $\CC$.

In the case of odd $K$-theory, we have the following characterisation of the Chern character using the homotopy theoretic model $K^{-1}(M) = [M,GL]$.
For a smooth map $g\colon M\to GL(n)$ the odd Chern character form is given by
\begin{equation}
\label{eqn:oddcherncharacterform}
Ch(g) = \sum_{j=0}^\infty \frac{-j!}{(2j+1)!} \bigg(\!\!-\frac{1}{2\pi i}\bigg)^{j+1} \tr\big(\underbrace{g^{-1}dg\wedge\dotsb\wedge g^{-1}dg}_{\text{$2j+1$ }}\big),
\end{equation}
which is a closed odd complex-valued form on $M$.
Moreover, the cohomology class of $Ch(g)$ depends only on the smooth homotopy class of $g$ \cite[Proposition 1.2]{G} so we have the odd Chern character map
\[
ch \colon K^{-1}(M) \lo H^{odd}(M;\CC)
\]
sending $[g]\mapsto [Ch(g)]$, which is again a group isomorphism after tensoring with $\CC$.

We have already established that the functor $\cK^{-1}$ defines odd $K$-theory, so our present aim is to give a natural representation of the odd Chern character on $\cK^{-1}$.
We do this using the following
\begin{definition}
Given an $\Omega$ vector bundle $\sE\to M$ equipped with a module connection $\triangle$ and Higgs field $\phi$, define the \emph{string form}
\[
s(\triangle,\phi) = {\int_\sone} \Ch(\nabla)
\]
where $\nabla$ is the framed connection on the caloron transform $E\to M\times S^1$ determined by $\triangle$ and $\phi$.
\end{definition}

It is clear that $s(\triangle,\phi)$ is a closed odd complex-valued form on $M$ that is additive with respect to direct sum. On the other hand, multiplicativity with respect to the tensor product is lost due to the integration over the circle. By virtue of its definition the cohomology class of $s(\triangle,\phi)$ is independent of the choice of $\triangle$ and $\phi$ and the string form is natural so its cohomology class is a characteristic class for $\Omega$ vector bundles that we call the \emph{(total) string class}.

We can express the string form entirely in terms of $\triangle$ and $\phi$ as follows.
Given locally trivialising sections $\check{\bm{\mathrm{e}}}_i \in \Gamma(U,\sE)$ we obtain corresponding locally trivialising sections $\bm{\mathrm{e}}_i \in \Gamma(U\x\sone,E)$.
A straightforward calculation yields
\begin{equation}
\label{eqn:caloroncurvature2}
R_{(x,\theta)}\bm{\mathrm{e}}_i  = \pr^\ast\!\sR_x\check{\bm{\mathrm{e}}}_i(\theta) +\big(\!\pr^\ast\!\triangle\phi \wedge d\theta\big){\!}_x \check{\bm{\mathrm{e}}}_i(\theta),
\end{equation}
where  $R$ is the curvature of the caloron transformed connection $\nabla$, $\sR$ is the curvature of $\triangle$ and 
\[
\triangle\phi(\hat X)\check{s} := \triangle_{\hat X} (\phi(\check{s})) - \phi(\triangle_{\hat X} \check s)
\]
is the \emph{Higgs field covariant derivative} in the direction of the vector field $\hat X$.

Computing the Chern character using the expression on the right hand side of \eqref{eqn:caloroncurvature2} and integrating over the circle, we have the following
\begin{proposition}
The string form of an $\Omega$ bundle $\sE$ is given explicitly in terms of $\triangle$ and $\phi$ by
\[
s(\triangle,\phi) =  \sum_{j=1}^\infty \frac{1}{(j-1)!}\bigg(\frac{1}{2\pi i}\bigg)^j \int_\sone\tr\big( \triangle\phi \wedge  \underbrace{\sR\wedge\dotsb \wedge \sR}_{\text{$j-1$}} \big)\,d\theta.
\]
\end{proposition}
\begin{remark}
\label{remark:piecewisestringchern}
This formula is reminiscent of the expression for string classes for principal loop group bundles \cite[Proposition 4.6]{MV}. To make this relationship precise, let us introduce the \emph{normalised symmetrised trace}
\begin{equation}
\label{traces}
\overline{\tr}_k(\xi_1,\dotsc,\xi_k) = \frac{1}{(k!)^2 (2\pi i)^k} \sum_{\sigma\in S_k} \tr\big( \xi_{\sigma(1)} \dotsb \xi_{\sigma(k)}\big)
\end{equation}
where $S_k$ is the symmetric group. This is now an invariant polynomial on $\mathfrak{gl}(n)$ of degree $k$. Given $\triangle$ and $\phi$ on $\sE$, let $\sA$ and $\Phi$ denote the corresponding connection and Higgs field on the frame bundle $\cF(\sE)$. The degree $2k-1$ piece of the string form on $\cF(\sE)$ is then exactly a string class
\begin{equation}
\label{eqn:stringformpiece}
s_{\overline{\tr}_k}(\sA,\Phi) = k \int_\sone \overline{\tr}_k \big(\nabla\Phi, \!\underbrace{\sF,\dots,\sF}_{\text{$k-1$ }} \! \big)\,d\theta
\end{equation}
where $\nabla\Phi = d\Phi + [\sA,\Phi] - \partial_\theta \sA$ and $\sF =d\sA +\tfrac{1}{2}[\sA,\sA] = \sR$ is the curvature of $\sA$.
\end{remark}

\begin{example}
On the classifying $\Omega$ bundle $\sE(n) \to GL(n)$ equipped with the module connection $\t_\a(n)$ and Higgs field $\phi(n)$, by \eqref{eqn:stringformpiece} and \cite[Proposition 4.11]{MV} we have
\begin{equation}
\label{eqn:canstringforms}
s(\t_\a(n),\phi(n)) =\sum_{k=1}^\infty \bigg(\!\!-\frac{1}{2}\bigg)^{k-1} \frac{k!(k-1)!}{(2k-1)!} \,\overline{\tr}_k\big(\Theta,\underbrace{[\Theta,\Theta],\dots,[\Theta,\Theta]}_{\text{$k-1$ }} \big)
\end{equation}
where $\Theta$ is the Maurer-Cartan form on $GL(n)$. 
Since each $\Omega$ bundle of rank $n$ is a pullback of $\sE(n) \to GL(n)$, the expression above gives a differential form representative for the universal string class.
\end{example}

The string form extends to a group homomorphism
\[
s\colon \cK^{-1}(M) \lo H^{odd}(M;\CC)
\]
by sending $\sE-\sF \mapsto [s(\t,\phi)-s(\tilde \t,\tilde\phi)]$ for any choice of module connections $\t, \tilde \t$ and Higgs fields $\phi,\tilde \phi$ on $\sE,\sF$ respectively. The following theorem shows that the string form $s$ represents the odd Chern character on $\cK^{-1}$.
\begin{theorem}
\label{theorem:stringchern}
The following diagram commutes
\[
\xy
(0,15)*+{\cK^{-1}(M)}="1";
(25,0)*+{H^\bullet(M;\CC)}="2";
(50,15)*+{K^{-1}(M)}="3";
{\ar^{s} "1";"2"};
{\ar^{\cong} "1";"3"};
{\ar_{ch} "3";"2"};
\endxy
\]
\end{theorem}
\begin{proof}
Recall that the isomorphism $\cK^{-1}(M) \to K^{-1}(M)$ sends $\sE - \ul{L\CC}^n$ to $ [g]$, with $g\colon M\to GL(n)\subset GL$ a smooth classifyung map for $\sE$.
Modulo exact forms, we then have $ch([g])$ given by \eqref{eqn:oddcherncharacterform}
\[
ch([g]) = \sum_{k=0}^\infty \frac{-k!}{(2k+1)!} \bigg(-\frac{1}{2\pi i}\bigg)^{k+1} \tr\big(\underbrace{g^{-1}dg\wedge\dotsb\wedge g^{-1}dg}_{\text{$2k+1$ }}\big).
\]

By \eqref{eqn:canstringforms} we have
\[
s(\t_\a(n),\phi(n)) = \sum_{k=1}^\infty \frac{-(k-1)!}{(2k-1)!} \bigg(\!\!-\frac{1}{2\pi i}\bigg)^{k}\tr\big(\underbrace{\Theta\wedge\dotsb\wedge \Theta}_{\text{$2k-1$ }} \big)
\]
on the classifying $\Omega$ bundle $\sE(n) \to GL(n)$.
Since $h^\ast \Theta = h^{-1}dh$ for any smooth map $h\colon M\to GL(n)$ and $\ul{L\CC}^n$ is classified by the constant map at the identity, by naturality of the string form we have the desired result.
\end{proof}

\subsection{The string potential} A hallmark of the Chern--Weil construction is that the cohomology class of the Chern character is independent of the choice of connection. The differential form representative, however, is sensitive to this choice and the dependence is captured by the Chern--Simons form, constructed as follows.

Starting with a complex vector bundle $E \to M$, for $t\in \I$ let $\gamma(t) = \nabla_t$ be a smooth path of connections on $E$ or equivalently, a connection $\nabla^\c$ on $E\x\I$ such that
\begin{equation}
\label{eqn:connectioncondition}
\nabla^\c_{V} \pr^\ast \!s = 0
\end{equation}
for any vector field $V$ that is vertical for the projection $\pr\colon M\x\I \to M$ and section $s \in \Gamma(E)$.
We write $\gamma'(t) = \nabla'_t = \cL_{\d_t}\nabla_t \in \Omega^1(M,\End(E))$ for the \lq\lq time derivative" at $t$ of the path $\c$, where $\cL_{\d_t}$ is the Lie derivative along the canonical vector field $\d_t$ in the $\I$ direction.
Then the Chern--Simons form is the odd complex-valued form on $M$ defined by
\begin{equation}
\label{eqn:chernsimons}
\CS(\c) = \sum_{j=1}^\infty \frac{1}{(j-1)!}\bigg(\frac{1}{2\pi i}\bigg)^j \int_0^1  \tr\big(\nabla'_t\wedge \underbrace{R_t\wedge\dotsb\wedge R_t}_{\text{$j-1$}}\big)
\end{equation}
where $R_t$ is the curvature of $\nabla_t$.
It is a standard fact that
\[
d\CS(\c) = \Ch(\nabla_1) - \Ch(\nabla_0)
\]
and furthermore we have that $\CS(\c_0) - \CS( \c_1) $ is exact for any two paths $\c_0$ and $ \c_1$ with the same endpoints \cite[Proposition 1.1]{SSvec}.
There is an alternative formulation that is often useful in calculations, namely let $\varsigma_t\colon M\hookrightarrow M\x\I$ denote the slice map $m\mapsto (m,t)$, then \cite[(1.8)]{SSvec}
\[
\CS(\c) = \int_0^1 \varsigma_t^\ast \imath_{\d_t} \Ch(\nabla^\c)
\]
where $\imath_{\d_t}$ denotes contraction along the vector field $\d_t$.

Proceeding by analogy in the odd case, we define the \emph{string potential} that captures the dependence of the string form on the particular choice of connective data.
Fix an $\Omega$ vector bundle $\sE\to M$, for $t\in \I$ let $\c(t) = (\t_t,\phi_t)$ be a smooth path of module connections and Higgs fields,  i.e. ~$\t^\c$ is a module connection on $\sE\x\I$ that satisfies the analogue of \eqref{eqn:connectioncondition} and $\phi^\gamma$ is a Higgs field on $\sE\x\I$.
\begin{definition} The \emph{string potential} of $\c$ is the even complex-valued form on $M$ defined by
\begin{equation}
\label{eqn:stringpotdef}
S(\c) := \int_0^1 \varsigma_t^\ast \imath_{\d_t} s(\t^\c,\phi^\c).
\end{equation}
\end{definition}
It is elementary to check that the total string potential is natural with respect to pullbacks.
Similarly to the Chern--Simons form, we have the following
\begin{proposition}
\label{prop:exactness}
For any smooth path $\c$ as above
\[
dS(\c) = s(\t_1,\phi_1) - s(\t_0,\phi_0).
\]
Moreover if $\c_0$ and $\c_1$ have the same endpoints then $S(\c_0) - S(\c_1)$ is exact. 
\end{proposition}

\begin{proof}
Using Cartan's formula ($ \cL= d\imath + \imath d$), Stokes' theorem and the fact that string forms are closed and natural we have the first result.

Next let $\gamma_0$ and $\gamma_1$ be two paths with the same endpoints and fix a smooth path $\Gamma$  between them, so that $\Gamma(t,i) = \gamma_i(t)$ and $\Gamma(i,s)=\overline \gamma_i$ is a constant map for $i=0,1$. 
As before, $\Gamma$ determines a module connection $\t^\Gamma$ on $\sE\x\I^2$ that vanishes on vectors vertical for $M\x\I^2\to M$ and a Higgs field $\phi^\Gamma$ on $\sE\x\I^2$.
We denote by $\varsigma_{(t,s)}$ the slice map $x\mapsto (x, t,s)$.
A similar calculation  to the first part then shows
$$
S(\gamma_1) - S(\gamma_0) = d \int_{\I^2}  \varsigma_{(t,s)}^\ast \big(  \imath_{\d_s} \imath_{\d_t} s(\t^\Gamma,\phi^\Gamma)\big),
$$
which completes the proof.
\end{proof}

We can give an explicit expression for the string potential as follows.
Fixing a smooth path $\c$, write
\[
\c'(t) = (\t'_t,\phi'_t)  \in \Omega^1(M,\End_{L\CC}(\sE)) \x L\mathfrak{gl}(n)
\]
where $n= \rk\sE$.
Let $\sR^\c$ denote the curvature form of $\t^\c$ on $M\x\I$ and $\sR_t$ the curvature of $\t_t$ on $M$. Write $\t^\c\phi^\c$ and $\t_t\phi_t$ for the Higgs field covariant derivatives of $\phi^\c$ and $\phi_t$ on $M\x\I$ and $M$ respectively.
For any vector field $X$ on $M\x\I$ tangential to the slice through $t$, we note that
\[
\sR^\c(\d_t,X) = \t^\c_{\d_t} \t^\c_{X} - \t^\c_{X} \t^\c_{\d_t} = \frac{d}{dt} \t^\c_X
\]
and
\[
\t^\c\phi^\c(\d_t) = \t^\c_{\d_t}\circ \phi^\c - \phi^\c\circ \t^\c_{\d_t} = \t^\c_{\d_t} \circ\phi^\c = \frac{d}{dt} \phi^\c .
\]
Thus it follows that   $\varsigma_t^\ast\imath_{\d_t}\sR^\c = \t'_t$ and $\varsigma_t^\ast\imath_{\d_t}\t^\c\phi^\c = \phi'_t$ and we have 
\begin{proposition}\label{string potential}

The string potential of a smooth path $\c$ of module connections and Higgs fields is given by
\[
S(\c) = \sum_{j=1}^\infty j \int_0^1\int_\sone \bigg[ (j-1) \overline{\tr}_j(\t'_t,\underbrace{\sR_t,\dotsc,\sR_t}_{\text{$j-2$}},\t_t\phi_t) + \overline{\tr}_j(\underbrace{\sR_t,\dotsc,\sR_t}_{\text{$j-1$}},\phi'_t)\bigg]
\]
\end{proposition}
\begin{proof}
This follows readily from the fact that
\[
s(\t,\phi) = \sum_{j=1}^\infty j \int_\sone\overline{\tr}_j(\t\phi,\underbrace{\sR,\dotsc,\sR}_{\text{$j-1$}})
\]
and from Definition \ref{eqn:stringpotdef} of the string potential $S(\c)$, together with the identities we just established above. 
\end{proof}

We conclude this section by recording a useful relationship between the string potential and the Chern--Simons form. It is not difficult to see that paths $\c$ of module connections and Higgs fields on an $\Omega$ vector bundle $\sE \to M$ correspond bijectively to paths $\c^c$ of framed connections on the caloron transform $E$. Then, as is straightforward to verify,
\begin{equation}\label{eqn:stringycs}
S(\c) = \int_0^1 \varsigma_t^\ast \imath_{\d_t} s(\t^\c,\phi^\c) = {\int_\sone}\CS(\c^c).
\end{equation}
The interested reader can also verify that this expression may be obtained directly using \eqref{eqn:caloroncurvature2}.

\subsection{The total string potential} 
The string potential depends on a pair of module connections and Higgs fields and descends to an even form on the base manifold. There exists another secondary invariant for $\Omega$ bundles $\sE$ that we call the \emph{total string potential}, which is associated to a single pair $(\t,\phi)$ and resides on the total space $\sE$. The construction is more lucid in the language of principal bundles and proceeds as follows.

Let $\pi\colon \cF(\sE)\to M$ denote the frame bundle of the $\Omega$ bundle $\sE$ and let $(\sA,\Phi)$ be the connection and Higgs field corresponding to $(\t,\phi)$. Recall that the fibre product $\cF(\sE)\times_M \cF(\sE)$ is canonically trivialised over $\cF(\sE)$ by the diagonal section $\xi\colon e \mapsto (e,e)$. This trivialisation singles out a trivial connection $\sA_\xi$ and trivial Higgs field $\Phi_\xi$ on the fibre product, and thus it is natural to consider the straight line path $\gamma\colon t\mapsto (1-t)(\sA_\xi, \Phi_\xi) + t(\xi^*\sA,\xi^*\Phi)$. The total string potential is defined by the usual string potential of this  line segment, 
\[S(\sA,\Phi) := \xi^\ast S(\gamma) \in \Omega^{2k-2}(\cF(\sE)).\]
Since the curvature and Higgs field covariant derivative of the trivial pair $(\sA_\xi, \Phi_\xi)$ vanish, we conclude from Proposition \ref{prop:exactness} that $dS(\sA,\Phi) = \pi^*s(\sA,\Phi)$.
\begin{proposition} The total string potential on $\cF(\sE)$ with connection $\sA$ and Higgs field $\Phi$ is given by
\begin{multline*}
S(\sA,\Phi) =  \sum_{j=1}^\infty  \int_\sone \Bigg[   \sum_{i=0}^{j-1} c_{i,j}   \overline{\tr}_j\Big(\Phi,  \underbrace{[\sA,\sA],\dotsc,[\sA,\sA]}_{\text{$i$ }},\underbrace{\sF,\dotsc,\sF}_{\text{$j-i-1$ }}  \Big)\\
 +  
 2\sum_{i=1}^{j-1} c_{i,j}   \overline{\tr}_j\Big(i[\sA,\Phi]-(i+j)\nabla\Phi,\sA, \underbrace{[\sA,\sA],\dotsc,[\sA,\sA]}_{\text{$i-1$ }},\underbrace{\sF,\dotsc,\sF}_{\text{$j-i-1$ }} \Big)\Bigg]
\end{multline*}
where the coefficients are
\[
c_{i,j} = \left(-\frac{1}{2}\right)^i\frac{j!(j-1)!}{(j+i)!(j-1-i)!}.
\]
\end{proposition}

\begin{proof} The result follows by direct calculation using the expression in Proposition \ref{string potential}. First we note that pulling back the string potential along the diagonal section $\xi$ eliminates the trivial pair $(\sA_\xi, \Phi_\xi)$, so we might as well compute with the path $t \mapsto t(\sA,\Phi)$. This gives $(\sA'_t, \Phi'_t) = (\sA, \Phi)$ and the associated curvature form  $\sF_t =  t\sF -\tfrac{t}{2}(1-t) [\sA,\sA]$ and Higgs field covariant derivative $\nabla\Phi_t = t\nabla \Phi - t(1-t)[\sA,\Phi]$. Inserting these into the formula for the string potential, applying the binomial expansion and integrating over the interval using $\int_0^1 t^{j-1}(1-t)^i dt= \frac{i!(j-1)!}{(j+i)!}$
we obtain the desired formula.
\end{proof}

\begin{corollary} \label{omega generators}
The total string potential restricted to any fibre $\cF(\sE)_m \cong \Omega GL(n)$ is given by
\[ \widehat \tau :=  -2 \sum_{j=1}^\infty c_{j-2,j} \int_\sone\overline{\tr}_j\Big(\gamma^{-1}\partial \gamma, \underbrace{[\Theta,\Theta],\dotsc,[\Theta,\Theta]}_{\text{$j-1$ }}  \Big), \]
where $\gamma \in \Omega GL(n)$ and $\Theta$ is the Maurer--Cartan form on $\Omega GL(n)$. The components of $\widehat \tau$ are the primitive generators for the cohomology ring $H^\bullet(\Omega GL(n),\RR)$.
\end{corollary}

\begin{proof} Let $\iota\colon \Omega GL(n) \to \cF(\sE)_m$ denote the isomorphism with the fibre over $m\in M$ and note that $\iota^* \sA = \Theta$ and $(\iota^* \Phi) (\gamma)= \gamma^{-1}\partial \gamma$. In particular, the restriction of the curvature  $\sF$ and Higgs field covariant derivative $\nabla \Phi$ to a fibre vanish, so the total string potential simplifies to 
\[
\iota^*S(\sA,\Phi) 
=- \sum_{j=1}^\infty 2c_{j-2,j} \int_\sone\overline{\tr}_j\left(\gamma^{-1}\partial \gamma , [\Theta,\Theta]^{j-1} \right),
\]
using ad-invariance of the $\overline{\tr}_j$ and the relation $(2j-1)c_{j-1,j} = -2c_{j-2,j}$. The components of this even form coincide with the transgression of the generators of $H^\bullet(GL(n),\RR)$ and it is well-known \cite[Appendix 4.11]{PS} that the latter generate the polynomial algebra $H^\bullet(\Omega GL(n), \RR)$.
\end{proof}

\begin{example}
In \cite{MS}, the authors consider the lifting bundle gerbe associated to the standard central extension 
\[1 \to S^1 \to \widehat{\Omega G} \to \Omega G \to 1 ,\]
where $G$ is a compact Lie group with a normalised Killing form $ \langle \cdot,\cdot\rangle$. They give a formula for the bundle gerbe curving 
\[
B = \frac{1}{2\pi i} \int_\sone  \langle\sF,\Phi\rangle - \tfrac{1}{2}\langle \sA,\d\sA\rangle ,
\]
which is a 2-form on the total space of an $\Omega G$-bundle $\sQ \to M$ equipped with a connection $\sA$ and Higgs field $\Phi$, and whose differential descends to a representative of the string class $s_3(\sA,\Phi) \in H^3(M,\ZZ)$. Computing the degree two component of the total string potential for $G=U(n)$  with $\langle \cdot,\cdot\rangle = - 8\pi^2 \overline{\tr}_2(\cdot,\cdot)$, we have
\begin{align*}
S_2(\sA,\Phi) &= -\frac{1}{8\pi^2} \int_\sone  \langle \Phi,\sF\rangle -\tfrac{1}{6} \langle \Phi,[\sA,\sA] \rangle - \tfrac{1}{3} \langle \sA,[\sA,\Phi] \rangle + \langle \sA,\nabla\Phi \rangle  \\
&= -\frac{1}{8\pi^2} \int_\sone 2   \langle \Phi,\sF\rangle - \langle \sA,\d\sA \rangle  - \langle A, d\Phi \rangle + \langle \Phi,d\sA \rangle \\
&=  \frac{1}{2\pi i} B + d \Big( \frac{1}{4\pi^2} \int_\sone \langle\sA,\Phi\rangle \Big).
\end{align*}
Thus, the total string potential recovers the curving of the lifting bundle gerbe up to an exact form. 
\end{example}

The total string potential is the odd analogue of the \emph{total} Chern--Simons form for a principal $GL(n)$-bundle $Q$ with connection $A$ with curvature $F$, defined as \cite[(3.5)]{CS}
\[CS(A) =  \sum_{j=1}^\infty  \sum_{i=0}^{j-1} c_{i,j}   \overline{\tr}_j\Big(A,\underbrace{[A,A],\dotsc,[A,A]}_{\text{$i$ }},\underbrace{F,\dotsc,F}_{\text{$j-i-1$ }} \Big) \! .\]
When $Q= \cF(E)$ is the caloron transform of the frame bundle $\cF(\sE)$ and the framed connection $A$ is determined by $\sA$ and $\Phi$, we have the relationship
\begin{equation}
\label{eqn:stringytotcs}
S(\sA,\Phi) = {\int_\sone}\CS(A) .
\end{equation}
This is the analogue of equation (\ref{eqn:stringycs}), although the interpretation is more subtle since $\CS(A)$ resides on the total space of $\cF(E)$ and integration over the circle is not straightforward in this case. However, recall that $A$ as defined in \eqref{eqn:finiteconnconstruct}  lifts to a basic form on $\cF(\sE)\times S^1\times GL(n) $. The connection $A$ appearing on the right hand side in (\ref{eqn:stringytotcs}) should therefore be viewed as the pullback of this basic form by the projection map onto  $\cF(\sE) \times S^1$, after which the $S^1$-integration makes sense.
 
In \cite[Proposition 3.15]{CS} it was shown that $CS(A)$, reduced mod $\ZZ$, defines an even  differential character on $M$. Indeed, this was one of the main motivations of the Cheeger--Simons model for differential cohomology $\check H^\bullet(M)$. By  \eqref{eqn:stringytotcs}, it follows that the mod $\ZZ$ reduction of the total string potential determines an odd differential character on $M$; the differential refinement of the string form.

\section{The $\Omega$ model for odd differential $K$-theory}\label{S:model}
The primary reason for dealing with $\Omega$ bundles is that they give convenient objects with which to describe odd $K$-theory.
Despite being extravagant in dimensions, working with $\Omega$ vector bundles is useful  since it allows one to phrase odd $K$-theory of the compact manifold $M$ entirely in terms of \emph{smooth} bundles based over $M$ and gives a bundle-theoretic interpretation of the identity $K^{-1}(M) = [M,B\Omega GL]$.
More importantly, by using module connections and Higgs fields we obtain a natural refinement to differential $K$-theory. 
The construction uses the string potential form in a role analogous to that played by the Chern--Simons form in the Simons--Sullivan construction of even differential $K$-theory \cite{SSvec}.

\subsection{Differential extensions}
First let us recall the framework of differential extensions due to Bunke--Schick \cite{BS3}, specialising to the case of complex $K$-theory.
Denoting by
\[
ch\colon K^\bullet(M) \lo H^\bullet(M;\CC)
\]
the Chern character of topological $K$-theory, we have the following
\begin{definition}
A \emph{differential extension of $K$-theory} is a contravariant functor $\check K^\bullet$ from the category of compact manifolds (possibly with corners) to $\ZZ_2$-graded abelian groups together with natural transformations
\begin{enumerate}
\item
$\Ch \colon \check K^\bullet (M) \to \Omega^\bullet_{d=0}(M;\CC)$ (the \emph{curvature});

\item
$I \colon \check K^\bullet(M) \to K^\bullet (M)$ (the \emph{underlying class}); and

\item
$a \colon \Omega^{\bullet-1}(M;\CC) /\im d \to \check K^\bullet(M)$  (the \emph{action of forms})
\end{enumerate}
such that
\begin{enumerate}
\item
the diagram
\[
\xy
(0,15)*+{\check K^\bullet(M)}="1";
(25,30)*+{K^\bullet(M)}="2";
(25,0)*+{\Omega^\bullet_{d=0}(M;\CC)}="3";
(50,15)*+{H^\bullet(M;\RR)}="4";
{\ar^{I} "1";"2"};
{\ar^{ch} "2";"4"};
{\ar^{\Ch} "1";"3"};
{\ar^{\deR} "3";"4"};
\endxy
\]
commutes, with $\deR$ the map induced by the de Rham isomorphism;
\item
$\Ch\circ a = d$, the exterior derivative; and
\item
the sequence
\[
K^{\bullet-1}(M) \xrightarrow{\;ch\;} \Omega^{\bullet-1}(M;\CC) /\im\,d  \xrightarrow{\;\;a\;\;} \check K^\bullet(M)  \xrightarrow{\;\;I\;\;} K^\bullet(M) \lo 0
\]
is exact.
\end{enumerate}
We denote the data of such a differential extension succinctly as the quadruple $(\check K^\bullet,\Ch,I,a)$.
\end{definition}


\begin{example}
The Simons--Sullivan model for even differential $K$-theory is the Grothendieck group completion of isomorphism classes of \emph{structured vector bundles}; that is vector bundles equipped with an equivalence class of connections defined by Chern--Simons exactness.
More precisely, if $E \to M$ is a smooth vector bundle we say that connections $\nabla_0$ and $\nabla_1$ for $E$ are \emph{equivalent} if there is a smooth path of connections $\c$ from $\nabla_0$ to $\nabla_1$ such that the Chern--Simons form of \eqref{eqn:chernsimons} is exact.
This defines an equivalence relation on the space of all connections on $E$, and a vector bundle $E \to M$ equipped with such an equivalence class $[\nabla]$ is a \emph{structured (vector) bundle over $M$}.

A cocycle in the Simons--Sullivan model of even differential $K$-theory, the underlying group of which we denote by $\edK$, is given by a virtual difference $\bm{E} - \bm{F}$ of structured vector bundles.
In fact, every element of $\edK$ may be written in the form $\bm{E} - \ul{\bm{n}}$ where $\ul{\bm{n}} = (\ul{\CC}^n,[d])$ is the trivial structured bundle of rank $n$.
As a differential extension, $\edK$ of course comes equipped with curvature, underlying class and action of forms maps $\check{ch}, \check I$ and $\check{a}$ respectively.
Briefly, $\check{ch}$ is given by computing the Chern character forms of \eqref{eqn:cherncharacterform}, $\check I$ is given by discarding the connective structure and $\check a$ may be characterised in a way completely analogous to the remark following Theorem \ref{odd diff K}.
The interested reader is referred to \cite{SSvec} for details.\end{example}

In \cite{BS2} it is shown that any two differential extensions of the even part of $K$-theory are isomorphic; in particular $\edK$ defines even differential $K$-theory.
We note that an \emph{isomorphism} of differential extensions is a natural isomorphism of the underlying functors that preserves the  curvature, underlying class and action of forms maps.
As clarified by Bunke--Schick, the axioms above do not uniquely determine differential extensions of \emph{odd} $K$-theory.
Indeed, to obtain uniqueness in odd degree we require either a multiplicative structure or an $\sone$-integration map:
\begin{definition}
\label{defn:intextension}
A differential extension $(\check K^\bullet,\Ch,I,a)$ of $K$-theory  has \emph{$\sone$-integration} if there is a natural transformation
\[
{\int_\sone} \colon \check K^\bullet (M\x \sone) \lo \check K^{\bullet-1}(M)
\]
compatible with the natural transformations $\Ch$ and $I$ and the $\sone$-integration maps on differential forms and on $K$.
If $\pr \colon M\x\sone \to M$ is the projection, we also require
\begin{enumerate}
\item
${\int_\sone} \pr^\ast x = 0$ for each $x \in \check K^\bullet(M)$; and

\item
${\int_\sone}(\id_M \x t)^\ast x = -{\int_\sone} x$ for all $x \in \check K^\bullet(M\x\sone)$, with $t \colon \sone\to\sone$ the (orientation-reversing) map given by complex conjugation. 
\end{enumerate}
An isomorphism of differential extensions with $\sone$-integration is also required to preserve the $\sone$-integration maps.
\end{definition}

\subsection{Structured $\Omega$ vector bundles}
By analogy with the structured vector bundles of Simons--Sullivan, we now introduce \emph{structured $\Omega$ vector bundles} as the basic ingredient underlying our model for odd differential $K$-theory.

Let $\sE \to M$ be an $\Omega$ vector bundle equipped with module connection $\triangle$ and Higgs field $\phi$. As established in Theorem \ref{theorem:stringchern}, the odd Chern character of $\sE$ is represented in terms of $\triangle$ and $\phi$ via the  string form and, using the normalised symmetrised traces (\ref{traces}), the string form can be written more elegantly as
\[
s(\triangle,\phi) = \sum_{j=1}^\infty j \int_\sone\overline{\tr}_j(\triangle \phi, \underbrace{\sR,\dotsc,\sR}_{\text{$j-1$ }}).
\]
Similarly, for any smooth path $\c(t) =  (\triangle_t,\phi_t)$ of module connections and Higgs fields on $\sE$, the string potential becomes
\[
S(\gamma) = \sum_{j=1}^\infty j \int_0^1\int_\sone \bigg[ (j-1) \overline{\tr}_j(\triangle'_t,\underbrace{\sR_t,\dotsc,\sR_t}_{\text{$j-2$ }},\triangle \phi_t) + \overline{\tr}_j(\underbrace{\sR_t,\dotsc,\sR_t}_{\text{$j-1$ }},\phi'_t)\bigg] .
\]

Note that there is a smooth path $\c$ connecting any pair $(\triangle_0,\phi_0)$ on $\sE$ to any other pair $(\triangle_1\,,\phi_1)$  and by Proposition \ref{prop:exactness} a different choice of path with the same endpoints amounts to a shift by an exact form.
We thus have a well defined map
\[
\cS(\triangle_0,\phi_0;\triangle_1,\phi_1) := S(\c) \mod\mbox{exact}
\]
which satisfies the transitivity relation
\[
\cS(\triangle_0,\phi_0;\triangle_2,\phi_2) = \cS(\triangle_0,\phi_0;\triangle_1,\phi_1) + \cS(\triangle_1,\phi_1;\triangle_2,\phi_2).
\]
This induces an equivalence relation on the space of module connections and Higgs fields by
\[
(\triangle_0,\phi_0) \sim (\triangle_1,\phi_1) \ \mbox{if and only if} \  \cS(\triangle_0,\phi_0;\triangle_1,\phi_1) = 0\mod\mbox{exact},
\]
and we call an equivalence class $[\t,\phi]$ a \emph{string datum} on $\sE$.
\begin{definition}
A \emph{structured $\Omega$ vector bundle} is a pair $\bm{\sE} = (\sE,[\triangle,\phi])$ where $\sE \to M$ is an $\Omega$ vector bundle and $[\triangle,\phi]$ is a string datum on $\sE$.
\end{definition}

By naturality of string potentials and string forms, for a smooth map $f \colon N \to M$ we may define the pull back of a structured $\Omega$ vector bundle $\bm{\sE} = (\sE,[\triangle,\phi])$ on $M$  by $f^\ast\bm{\sE}:= (f^\ast\sE,f^\ast[\triangle,\phi])$ where $f^\ast[\t,\phi] := [f^\ast\t,f^\ast\phi]$. 
There is an obvious notion of \emph{isomorphism} of structured $\Omega$ vector bundles. 
It is easy to verify that the direct sum operation extends to string data, so setting
\[
\bm{\sE} \oplus \bm{\sF} := (\sE\oplus\sF,[\triangle\oplus\tilde\triangle,\phi\oplus\tilde\phi])
\]
gives a well defined operation on structured $\Omega$  bundles.

It is important for our purposes to understand how structured $\Omega$ vector bundles behave under smooth homotopies.
Suppose that $f_t \colon N \to M$ is a family of smooth maps depending smoothly on the parameter $t \in\I$ and we are given an $\Omega$ vector bundle $\sE \to M$ equipped with module connection $\triangle$ and Higgs field $\phi$.
By taking the caloron transform and using parallel transport along the family of curves $\rho_{x,\theta} \colon \I \to M \x \sone$, $\rho_{x,\theta}(t) := (f_t(x),\theta)$, we obtain isomorphisms $f_0^\ast\sE \cong f_t^\ast\sE$ for each $t\in\I$.
Abusing notation slightly by omitting explicit reference to these isomorphisms and inserting into \eqref{eqn:stringpotdef} we get
\begin{equation}
\label{eqn:homstringform}
\cS(f_0^\ast\triangle,f_0^\ast\phi;f_1^\ast\triangle,f_1^\ast\phi) = \int_0^1 f_t^\ast \imath_{\dot\rho_x(t)} s(\triangle,\phi)\,dt \mod\mbox{exact},
\end{equation}
where $\dot\rho_x(t)$ is the tangent to the curve $\rho_x(t) := f_t(x)$ at $t\in\I$.

\begin{definition}
Let  $\struct(M)$ denote the set of isomorphism classes of structured $\Omega$ vector bundles over $M$. 
Direct sum makes $\struct(M)$ an abelian semigroup and the assignment $\struct\colon M\mapsto \struct(M)$ defines a contravariant functor.

\end{definition}
\begin{remark}
We shall usually denote an element of $\struct(M)$ by $\bm{\sE}$, rather than the technically correct $[\bm{\sE}]$, to avoid an excess of notation.
\end{remark}
\begin{example}
The \emph{trivial structured $\Omega$ bundle of rank $n$} is $\ul{\bm{\sL\CC}}^n := (\ul{L\CC}^n,[\de,\d])$.
\end{example}

\begin{example}
Recall the bundles $\sE(n) \to GL(n)$ equipped with the Higgs fields $\phi(n)$ and module connections $\t_\alpha(n)$.
As remarked in Example \ref{example:universalmodconn}, the connection $\t_\alpha(n)$ depends on the choice of a smooth funtion $\alpha \colon \RR \to \I$.
At the level of string data, however, we claim that $[\t_\alpha(n),\phi(n)]$ is independent of the choice of $\alpha$.

To see this, take any two choices of $\alpha_0$, $\alpha_1$ of such smooth functions and consider the line segment $\c$ from $(\t_{\a_0}(n),\phi(n))$ to $(\t_{\a_0}(n),\phi(n))$.
The module connection determined by a point $s$ in this line segment then corresponds to the $\Omega GL(n)$-connection
\[
\widehat \sA = \Theta - \big( s\a_1 + (1-s)\a_0 \big)  \ad(p^{-1}) \pi^\ast\widehat\Theta
\]
on $PGL(n) \x\I$, where $s \in \I$.
Writing $\widehat\a = s\a_1 + (1-s)\a_0$, a simple calculation gives that the curvature of $\widehat\sA$ is
\[
\sF = \tfrac{1}{2} \big(\widehat\a^2-\widehat\a\big)\ad(p^{-1})\big[\pi^\ast\widehat\Theta,\pi^\ast\widehat\Theta \big].
\]
The Higgs field $\phi(n)$ is determined by the canonical Higgs field $\Phi$ on $PGL(n)$ (as in Example \ref{example:universalHF}), which satisfies
\[
\nabla\Phi = \d\widehat\alpha\ad(p^{-1})\pi^\ast\widehat\Theta
\]
on $PGL(n)\x\I$.
The string form associated to this data is then
\[
s(\t^\c,\phi^\c)= \sum_{j=1}^\infty \bigg(\!\!-\frac{1}{2}\bigg)^{j-1} \frac{j!(j-1)!}{(2j-1)!} \,\overline{\tr}_j\big(\Theta,\underbrace{[\Theta,\Theta],\dots,[\Theta,\Theta]}_{\text{$j-1$ }} \big).
\]
Notice that this has no component in the $\I$ direction, so by \eqref{eqn:stringpotdef} the string potential  is $S(\c) = 0$.
We call $\bm{\sE}(n) = (\sE(n),[\t_\a(n),\phi(n)])$ the \emph{canonical} structured $\Omega$ bundle of rank $n$.
In light of this fact, we will usually not specify a choice of $\a$ when referring to $\bm{\sE}(n)$. 
\end{example}

The following sub-semigroups of $\struct(M)$ are important in the sequel\\
\begin{itemize}
\item
 $\struct_0(M) := \{(\sE,[\triangle,\phi]) \in\struct(M) \mid \text{$\sE$ is trivial} \}$ is the semigroup of \emph{topologically trivial} structured $\Omega$ vector bundles;\\

\item
 $\struct_{cl}(M) := \{ g^\ast\bm{\sE}(n) \mid \text{$g \colon M \to GL(n)$ is smooth for some $n$} \}$ is the semigroup of structured $\Omega$ vector bundles that are \emph{classified} by the $\bm{\sE}(n)$;\\

\item
 $\struct_T(M) := \{(\sE,[\triangle,\phi]) \in\struct(M) \mid \text{$\sE \oplus \ul{L\CC}^n$ is trivial for some $n$} \}$ is the semigroup of \emph{stably trivial} structured $\Omega$ vector bundles; \\

\item
 $\struct_F(M) := \{\bm{\sE} \in\struct(M) \mid \text{$\bm{\sE}\oplus \ul{\bm{\sL\CC}}^n$ is trivial for some $n$} \}$ is the semigroup of \emph{stably flat} structured $\Omega$ vector bundles.\\
\end{itemize}
Note that $\struct_F(M)$ and $\struct_0(M)$ are both sub-semigroups of $\struct_T(M)$ and all  assignments $M \mapsto \struct_\ast (M)$ are functorial.

Before defining the $\Omega$ model for odd differential $K$-theory, we must verify that every element of $\struct(M)$ has an \emph{inverse}, that is for each $\bm{\sE}$ there is some $\bm{\sF}$ such that $\bm{\sE} \oplus\bm{\sF}  \cong \ul{\bm{\sL\CC}}^n$.
We shall do this in two steps by first showing explicitly that every element of $\struct_{cl}(M)$ has an inverse and then demonstrating that modulo $\struct_0(M)$ every element of $\struct(M)$ lives in $\struct_{cl}(M)$.
The intuition here is that we can \lq\lq cancel off'' the geometric data on $\bm{\sE}$ by a topologically trivial structured $\Omega$ vector bundle to obtain a pullback of $\bm{\sE}(n)$, for which we have explicit inverses.

\begin{lemma}
\label{lemma:classinverse}
Each element of $\struct_{cl}(M)$ has an inverse in $\struct_{cl}(M)$.
\end{lemma}
\begin{proof}
Take any $g^\ast\bm{\sE}(n) \in \struct_{cl}(M)$.
We show that $(g^{-1})^\ast\bm{\sE}(n)$ is an inverse to $g^\ast\bm{\sE}(n)$, with $g^{-1} \colon M\to GL(n)$ the map $x \mapsto g(x)^{-1}$. 
To see this, observe that $g^\ast\bm{\sE}(n) \oplus (g^{-1})^\ast\bm{\sE}(n)$ is canonically isomorphic to the pullback of $\bm{\sE}(2n)$ by the block sum map
\[
g\oplus g^{-1} \colon x\longmapsto
\begin{bmatrix}
g(x) & 0\\
0 & g(x)^{-1}\\
\end{bmatrix}
\]
where each entry is an $n\x n$-matrix.
For $t \in [0,\tfrac{\pi}{2}]$ define the map $X_t \colon M \to GL(2n)$ by
\[
x\longmapsto 
\begin{bmatrix}
g(x) & 0\\
0 & 1 \\
\end{bmatrix} 
\begin{bmatrix}
\cos t & -\sin t\\
\sin t & \cos t \\
\end{bmatrix}
\begin{bmatrix}
1 & 0\\
0 & g(x)^{-1}\\
\end{bmatrix}
\begin{bmatrix}
\cos t & \sin t\\
-\sin t & \cos t \\
\end{bmatrix}.
\]
The family of maps $X_t$ gives a smooth homotopy from $g\oplus g^{-1}$ to the constant map $\id\colon x\mapsto \id\in GL(2n)$ and, as  in \cite[Lemma 3.6]{TWZ}, we have
\[
\tr\big( X_t^{-1} \d_t X_t \cdot (X_t^{-1} d X_t)^{2j}\big) = 0
\]
for each $j\geq 0$. Denoting by $\triangle$ and $\phi$ the pullback module connection and Higgs field on $(g\oplus g^{-1})^\ast\bm{\sE}(2n)$ and using \eqref{eqn:homstringform}, we conclude that
\begin{multline*}
\cS(\triangle,\phi;\de,\d) = \int_0^{\tfrac{\pi}{2}} X_t^\ast \imath_{\d_t X_t(x)} \Bigg[ \sum_{j=0}^\infty \frac{-j!}{(2j+1)!}\bigg(-\frac{1}{2\pi i} \bigg)^{j+1}\tr\big(\Theta^{2j+1} \big)\Bigg]dt\\
=  \int_0^{\tfrac{\pi}{2}}  \sum_{j=0}^\infty \frac{-j!}{(2j)!}\bigg(-\frac{1}{2\pi i} \bigg)^{j+1}\tr\big(X_t^{-1}\d_t X_t \cdot (X_t^{-1} dX_t)^{2j} \big) \,dt= 0 \mod\mbox{exact}
\end{multline*}
since $X_t^\ast \imath_{\d_t X_t(x)} \Theta = X_t^{-1}\d_t X_t$ and $X_t^\ast\Theta = X_t^{-1} dX_t$. Thus we have 
\[
g^\ast \bm{\sE}(n) \oplus (g^{-1})^\ast \bm{\sE}(n)  = (g\oplus g^{-1})^\ast\bm{\sE}(2n) = \ul{\bm{\sL\CC}}^{2n}
\]
as required.
\end{proof}

If $\ul{L\CC}^n \to M$ is a trivial $\Omega$ vector bundle, for any choice of module connection $\triangle$ and Higgs field $\phi$ on $\ul{L\CC}^n$ let $\gamma_{\triangle,\phi}$ denote the straight line path from the trivial pair $(\de,\d)$  to the pair $(\triangle,\phi)$.
Define a map
\[
\sS \colon \struct_0(M) \lo \Omega^{even}(M;\CC) /\im d 
\]
by sending $(\sE,[\triangle,\phi]) \longmapsto S(\gamma_{\triangle,\phi}) \!\mod \mbox{exact}$.
It is straightforward to check that $\sS$ is a well defined semigroup homomorphism and determines a natural transformation of functors.
\begin{theorem}\label{surjective}
The homomorphism $\sS$ is surjective.
\end{theorem}

\begin{proof}[Proof sketch]
The proof goes along the same lines as \cite[Proposition 2.6]{SSvec},  adapted to the case of structured $\Omega$ bundles. 
We begin by proving the result for $M = \RR^n$.
Consider the trivial $\Omega$ line bundle $\ul{L\CC} \to \RR^n$, which has caloron transform the trivial line bundle $\ul{\CC} \to \RR^n\x\sone$ with its canonical framing over $\RR^n_0:=\RR^n \x\{0\}$.

A framed connection on $\ul{\CC}$ corresponds to a complex-valued $1$-form $\omega$ that vanishes when pulled back to $\RR^n_0$.
Via the caloron correspondence we may view $\omega$ as the image of some module connection $\triangle$ and Higgs field $\phi$ on $\ul{L\CC}$ under the caloron transform, in which case
\[
\sS(\ul{L\CC},[\triangle,\phi]) = {\int_\sone} \CS(t\omega) 
= \sum_{j=1}^\infty \frac{1}{j!}\bigg(\frac{1}{2\pi i}\bigg)^{j}\,{\int_\sone} \omega\wedge d\omega^{j-1}\mod\mbox{exact}.
\]
Now pick any $f \in C^\infty(M;\CC)$ and set $\omega = ifd\theta$, noting that $\omega$ satisfies the pullback requirement so corresponds to a connection in the image of the caloron transform.
Then
\[
\sS(\ul{L\CC},[\triangle,\phi]) = \frac{1}{2\pi i}{\int_\sone} \omega = f.
\]
Thus we have that $\Omega^0(\RR^n;\CC) \subset \im \sS$.
Proceeding by induction, suppose for $k>0$ that
\[
\bigoplus_{i=0}^{k-1} \Omega^{2i}(\RR^n;\CC) / \im d\subset \im \sS.
\]
Write $\alpha = x_1dx_2 +\dotsb + x_{2k-1}dx_{2k}$ and $\beta = d\alpha = dx_1\wedge dx_2+ \dotsb + dx_{2k-1}\wedge dx_{2k}$ and let $\rho \colon \sone \to \RR$ be a smooth function such that $\rho(0) = 0$ and $\int_\sone \rho^k = (2\pi)^{k+1}$.
Then for any $f \in C^\infty(M)$ we set $\omega  = i\rho\,\alpha + ifd\theta$, noting as before that this is a framed connection on $\ul{\CC}$.
By a straightforward calculation we have
\begin{align*}
\bigg(\frac{1}{i}\bigg)^{k+1}\omega\wedge d\omega^k &= f \rho^k \beta^k\wedge d\theta + k\rho^k \alpha\wedge df\wedge \beta^{k-1}\wedge d\theta\\
&= (k+1)! f \rho^k dx_1\wedge\dotsb\wedge dx_{2k} \wedge d\theta \mod\mbox{exact}.
\end{align*}
and hence
\begin{multline*}
\sS(\ul{L\CC},[\triangle,\phi]) 
=  fdx_1\wedge \dotsb \wedge dx_{2k} +
\sum_{j=1}^{k} \frac{1}{j!}\bigg(\frac{1}{2\pi i}\bigg)^{j}\,{\int_\sone} \omega\wedge d\omega^{j-1}\mod\mbox{exact}.
\end{multline*}
By induction we have that
\[
\sS(\bm{\sE}) = - \sum_{j=1}^{k} \frac{1}{j!}\bigg(\frac{1}{2\pi i}\bigg)^{j}\,{\int_\sone} \omega\wedge d\omega^{j-1}\mod\mbox{exact}
\]
for some $\bm{\sE} \in \struct_0(M)$.
Hence, writing $\bm{\sL} := (\ul{L\CC},[\triangle,\phi])$ we conclude that
\[
\sS(\bm{\sE} \oplus \bm{\sL}) = fdx_1\wedge\dotsb\wedge dx_{2k} \mod\mbox{exact}.
\]
Since every $2k$-form on $\RR^n$ is a sum of such terms and $\sS$ is a homomorphism, we obtain the result for $\RR^n$.

For an arbitrary compact manifold $M$, we choose an embedding $\imath \colon M \to \RR^n$.
The pullback $\imath^\ast$ is surjective on forms and clearly $\imath^\ast\struct_0(\RR^n) \subset \struct_0(M)$, so the result follows by the  naturality of $\sS$.
\end{proof}

Combining Lemma \ref{lemma:classinverse} and Theorem \ref{surjective} we can finally prove
\begin{theorem}
\label{theorem:inversesexist}
Each element of $\struct(M)$ has an inverse.
\end{theorem}
\begin{proof}
For any arbitrary element $\bm{\sE} = (\sE,[\triangle,\phi]) \in\struct(M)$ with $\rk\sE = n$, pick a representative pair $(\triangle,\phi)$ for the string datum.
Recall that the Higgs field holonomy $\hol_\phi \colon M \to GL(n)$ is a smooth classifying map for $\sE$ that preserves the Higgs field. Writing $g = \hol_\phi$ for brevity, we have therefore $g^\ast{\sE}(n) \cong \sE$. Consider the straight line path $\gamma$ on $\sE$ connecting the pullback pair $(g^\ast\triangle(n),\phi(n)=\phi)$ to the original pair $(\triangle,\phi)$.

By Theorem  \ref{surjective} we can alway find $\bm{\sF}\in \struct_0(M)$ such that $\sS(\bm{\sF}) = -S(\gamma)$ mod exact.
Without loss of generality, $\bm{\sF} = (\ul{L\CC}^k,[\tilde\triangle,\tilde\phi])$ so that $S(\gamma_{\tilde \triangle,\tilde\phi}) = -S(\gamma)$ mod exact.
Then on $\sE \oplus \ul{L\CC}^k$ we have that $\gamma \oplus \gamma_{\tilde\triangle,\tilde\phi}$ is a smooth path from the pair $(g^\ast\triangle(n)\oplus\de,\phi\oplus \d)$ to the pair $(\triangle\oplus\tilde\triangle,\phi\oplus\tilde\phi)$ such that
\[
S(\gamma\oplus \gamma_{\tilde\triangle,\tilde\phi}) = S(\gamma) + S(\gamma_{\triangle',\phi'}) = 0 \mod\mbox{exact}.
\]
Recalling the notation $\id \colon M\to GL$ for the constant map at the identity, we thus have that $\bm{\sE} \oplus\bm{\sF} = g^\ast\bm{\sE}(n) \oplus \id^\ast\!\bm{\sE}(k) = (g\oplus \id)^\ast\bm{\sE}(n+k)$ in $\struct(M)$ and the result follows from Lemma \ref{lemma:classinverse}.
\end{proof}

\subsection{The $\Omega$ model}
Let $K(\struct(M))$ denote the Grothendieck group completion of the semigroup $\struct(M)$ and define the \emph{rank homomorphism}
\[
\rk \colon K(\struct(M)) \lo \ZZ
\]
that sends a formal difference of structured $\Omega$ bundles $\bm{\sE} - \bm{\sF}$ to its virtual rank $\rk\sE - \rk\sF \in\ZZ$. We define
\[
\dK(M):= \ker\rk
\]
so that elements of $\dK(M)$ are precisely virtual structured $\Omega$ vector bundles of rank zero; this is the \emph{$\Omega$ model}. There are a few immediate elementary consequences of this definition and Theorem \ref{theorem:inversesexist}, namely
\begin{itemize}
\item
every element of $\dK(M)$ is of the form $\bm{\sE} - \ul{\bm{\sL\CC}}^n$ where $n = \rk\sE$;

\item
$\bm{\sE} - \bm{\sF} = 0$ in $\dK(M)$ if and only if $\bm{\sE}$ and $\bm{\sF}$ are \emph{stably isomorphic},  \emph{i.e.} 
\[
\bm{\sE} \oplus \ul{\bm{\sL\CC}}^n = \bm{\sF} \oplus \ul{\bm{\sL\CC}}^n
\]
for some $n$; and hence

\item
$\bm{\sE} - \ul{\bm{\sL\CC}}^n = 0$ in $\dK(M)$ if and only if $\bm{\sE} \in \struct_F(M)$.
\end{itemize}

By definition of string data, the string form map $S \colon \struct(M) \to \Omega_{d=0}^{odd}(M;\CC)$ that sends $(\sE,[\triangle,\phi]) \mapsto s(\triangle,\phi)$ is a well defined semigroup homomorphism.
After passing to the group completion, we have the induced map
\[
S \colon \dK(M) \lo \Omega_{d=0}^{odd}(M;\CC)
\]
that sends $\bm{\sE} - \bm{\sF} \mapsto S(\bm{\sE}) - S(\bm{\sF})$.
There is also the natural surjection  
\[
\check I \colon \dK(M) \lo \cK^{-1}(M)
\]
that discards the connective data, giving the commuting diagram of natural homomorphisms
\[
\xy
(0,15)*+{\dK(M)}="1";
(25,30)*+{\cK^{-1}(M)}="2";
(25,0)*+{\Omega^{odd}_{d=0}(M;\CC)}="3";
(50,15)*+{H^\bullet(M;\CC)}="4";
{\ar^{\check I} "1";"2"};
{\ar^{ch} "2";"4"};
{\ar^{S} "1";"3"};
{\ar^{\deR} "3";"4"};
\endxy
\]
where $ch$ is the odd Chern character map as in Theorem \ref{theorem:stringchern}.
It is clear from this presentation that $S$ and $\check I$ may be viewed respectively as the curvature and underlying class maps of a differential extension of $\cK^{-1}$.

\subsection{The action of forms}
In order to obtain the action of forms on $\dK$ we use the map $\sS$ of Theorem \ref{surjective}, following ideas of Simons--Sullivan \cite[Section 2]{SSvec}.
\begin{definition}
A pair $(\triangle,\phi)$ on $\sE \to M$ is \emph{Flat} if the corresponding caloron transformed connection $\nabla$ on $E \to M\x\sone$ has trivial holonomy around any loop in $M\x\sone$.
By \eqref{eqn:caloroncurvature} this implies that $\t$ has curvature $\sR = 0$ and the Higgs field covariant derivative $\triangle \phi = 0$.
Moreover, since $\nabla$ admits global parallel sections, via the caloron correspondence we may identify $\sE$ with a trivial bundle equipped with the trivial pair $(\de,\d)$.
\end{definition}

Recall the  odd degree complex-valued form
\[
\tau = \sum_{j=0}^\infty \frac{-j!}{(2j+1)!} \bigg( \!-\!\frac{1}{2\pi i} \bigg)^{j+1} \tr\big(\Theta_{GL}^{2j+1} \big)
\]
on the stable general linear group $GL$. This can be transgressed to the based loop group $\Omega GL$ in the usual way; if $\ev \colon \Omega GL \x\sone \to GL$ is the evaluation map then it follows that
$\ev^\ast\Theta_{GL} = \Theta_{\Omega GL} + \Phi \,d\theta$, where $\Phi (\c) = \c^{-1}\d\c$ for $\c\in \Omega GL$,
so we obtain the closed  even degree complex-valued form
\[
\widehat\tau = {\int_\sone} \ev^\ast\tau =  \sum_{j=0}^\infty \frac{-j!}{(2j)!} \bigg( \!-\!\frac{1}{2\pi i} \bigg)^{j+1}\int_\sone \tr\big(\Phi\cdot \Theta_{\Omega GL}^{2j} \big)
\]
on $\Omega GL$. It is well known that the cohomology $H^\ast(\Omega GL;\CC)$ is a polynomial ring primitively generated by the components of  $\widehat\tau$ (cf. Corollary \ref{omega generators})

Define the space of closed  even-degree forms
\[
\wedge_{\Omega GL}(M) := \left\{ G^\ast \widehat\tau \mid \text{$G\colon M \to \Omega GL$ is smooth} \right\}
\]
with group structure given by $(G\oplus H)^\ast\widehat\tau = G^\ast\widehat\tau+ H^\ast\widehat\tau$ and $(G^{-1})^\ast\widehat\tau = -G^\ast\widehat\tau$.
By Bott periodicity $\Omega GL \simeq BGL\x\ZZ$, we may identify even $K$-theory with smooth homotopy classes of maps $M\to \Omega GL$, in which case the Chern character on $K^0(M)$ sends $[G] \mapsto [G^\ast\widehat\tau]\in H^\bullet(M;\CC)$.
This gives an identification of $\wedge_{\Omega GL}(M) \mod\mbox{exact}$ with the space of all even Chern characters on $M$.

\begin{proposition}
If $(\triangle,\phi)$ and $(\tilde \triangle,\tilde\phi)$ are any two Flat pairs on $\sE$, then
\[
\cS(\triangle,\phi;\tilde \triangle,\tilde\phi) \in \wedge_{\Omega GL}(M) \mod \mathrm{exact}.
\]
\end{proposition}
\begin{proof}
Let $\nabla$ and $\tilde \nabla$ be the corresponding caloron transformed connections.
Since $\nabla$ and $\tilde \nabla$ both have trivial holonomy we may without loss of generality take $\nabla = d$ and write $\tilde \nabla = g^{-1}\circ d \circ g$ for some smooth map $g \colon M\x\sone \to GL(n)$ viewed as an automorphism of the trivial bundle, with $n = \rank E$. The path of connections $\c(t) = d +tg^{-1} dg$ has curvature $ (t^2-t) g^\ast \big(\Theta\wedge\Theta\big)$, 
with $\Theta$ the Maurer--Cartan form on $GL(n)$.
It follows that the Chern--Simons form associated to $\c$ is 
\[
\CS(\c) 
=\sum_{j=1}^\infty \frac{-(j-1)!}{(2j-1)!} \bigg(\!-\!\frac{1}{2\pi i}\bigg)^j g^\ast \tr\big(\Theta^{2j-1} \big).
\]
Thus
\[
\cS(\triangle,\phi;\tilde \triangle,\tilde\phi) = \sum_{j=0}^\infty \frac{-j!}{(2j+1)!} \bigg(\!-\!\frac{1}{2\pi i}\bigg)^{j+1} {\int_\sone} g^\ast \tr\big(\Theta^{2j+1} \big) \mod\mbox{exact}.
\]
Denoting by $G\colon M \to \Omega GL(n)$ the map $G(m)(\theta) := g(m,\theta)$, we have 
\[
g^\ast\Theta_{(m,\theta)} = G^\ast\Theta_m + G^\ast\Phi(m) \,d\theta,
\]
where the $\Theta$ on the right hand side denoting the Maurer--Cartan form on $\Omega GL(n)$.
Plugging this into the above expression for $\cS(\triangle,\phi;\tilde \triangle,\tilde\phi)$ gives the result.
\end{proof}

\begin{remark}\label{all flat}
We note that by the above, for any smooth map $G \colon M \to \Omega GL(n)$, the straight line path $\gamma$ on the trivial $\Omega$ vector bundle $\ul{L\CC}^n$ from the trivial pair $(\de,\d)$ to the Flat pair $(G^{-1}\de G,G^{-1}\d G)$ yields $S(\gamma) = G^\ast\widehat\tau$ on the nose.
\end{remark}
Let $\widehat\wedge_{\Omega GL}(M):= \wedge_{\Omega GL}(M) + d\Omega^{odd}(M;\CC)$.
We extend the homomorphism $\sS$ of Theorem \ref{surjective} to a map $\struct_T(M) \to \Omega^{even}(M;\CC)/\widehat\wedge_{\Omega GL}(M)$ as follows.
If $\bm{\sE} = (\sE,[\triangle,\phi]) \in \struct_T(M)$  choose trivial bundles $\sF$ and $\sH$ such that $\sE \oplus \sF \cong \sH$.
Pick any Flat pairs $(\de,\d)$, $(\tilde \de,\tilde\d)$ on $\sH$ and $\sF$ respectively and set
\[
\sS(\bm{\sE}) := \cS(\de,\d;\triangle\oplus\tilde\de,\phi\oplus\tilde \d) \mod \widehat\wedge_{\Omega GL}(M).
\]
It is straightforward to verify that $\sS$ is a well defined semigroup homomorphism and is surjective on $\Omega^{even}(M;\CC)/ \widehat\wedge_{\Omega GL}(M)$ by Theorem \ref{surjective}.
However 

\begin{lemma}
The kernel of $\sS$ is precisely $\struct_F(M)$.
\end{lemma}
\begin{proof}
By definition we have $\struct_F(M) \subset \ker\sS$.
For the converse, take $\bm{\sE} \in \ker\sS$ so that
\[
\sS(\bm{\sE}) = \cS(\de,\d;\triangle\oplus\tilde\de,\phi\oplus \tilde \d) = G^\ast\widehat\tau \mod\mbox{exact}
\]
with $(\sH,[\de,\d] ) $ and $(\sF,[\tilde \de,\tilde \d])$ as above.
The caloron transform of $\bf \sH$ is a trivial bundle $\ul{\CC}^n \to M\x\sone$ with its trivial connection $d$.
Defining $g \colon M\x\sone \to GL(n)$ by $g(m,\theta) =G(m)(\theta)$, set $\nabla := g^{-1}\circ d \circ g$ and let $(\hat\de,\hat\d)$ denote the corresponding Flat pair on $\sH$. As in Remark \ref{all flat} we have $\cS(\de,\d,\hat\de,\hat\d) = G^\ast\widehat\tau \mod\mbox{exact}$ and so
\[
\cS(\hat\de,\hat\d;\triangle\oplus\tilde\de,\phi\oplus\tilde\d) = 0 \mod\mbox{exact},
\]
which shows that $\bm{\sE}$ is stably flat.
\end{proof}

\begin{corollary}
The map $\sS$ induces a semigroup isomorphism
\[
\sS \colon \struct_T(M)/\struct_F(M) \lo \Omega^{even}(M;\CC)/\widehat\wedge_{\Omega GL}(M)
\]
and hence $\struct_T(M)/\struct_F(M)$ is a group.
\end{corollary}

Note that we may identify $\struct_T(M)/\struct_F(M)$ with the kernel of the map $\check I$ as follows.
Let
\[
\jmath \colon \struct_T(M) /\struct_F(M) \lo \ker \check I
\]
be the map $\{\bm{\sE}\} \mapsto \bm{\sE} - \ul{\bm{\sL\CC}}^n$, recalling that $\sE - \ul{L\CC}^n = 0$ in $\cK^{-1}(M)$ if and only if $\sE$ is stably trivial.
Moreover, it is clear that $\jmath(\{\bm{\sE}\}) = 0$ if and only if $\bm{\sE}$ is stably flat and also that $\jmath$ is surjective, so by precomposing $\jmath$ by $\sS^{-1}$ we obtain an isomorphism
\[
\Omega^{even}(M;\CC)/\widehat\wedge_{\Omega GL}(M) \simto \ker\check I
\]
and hence an injection $\imath \colon \Omega^{even}(M;\CC)/\widehat\wedge_{\Omega GL}(M) \to \dK(M)$.
As noted previously, the even Chern character may be represented by pullbacks of the class $[\widehat\tau]\in H^\bullet(\Omega GL;\CC)$ and so it follows that $\widehat\wedge_{\Omega GL}(M) = \im ch$.
We can thus define the action of forms $\check a$ as the composition
\[
\Omega^{even}(M;\CC)/\im d \xrightarrow{\;\;\pr\;\;} \big( \Omega^{even}(M;\CC)/\im d \big) \big/ \im ch \xrightarrow{\;\;\imath\;\;} \dK(M)
\]
from which it follows that the sequence
\[
K^{0}(M) \xrightarrow{\;\;ch\;\;} \Omega^{even}(M;\CC) /\im\,d  \xrightarrow{\;\;\check a\;\;} \dK(M)  \xrightarrow{\;\;\check I\;\;} \cK^{-1}(M) \lo 0
\]
is exact. To verify that $S \circ \check a  = d$, we note that for any stably trivial $\bm{\sE} = (\sE,[\triangle,\phi])$ 
\[
d\sS(\{\bm{\sE}\}) = s(\triangle,\phi) = S \circ \jmath\, (\{\bm{\sE}\}).
\]
Thus for any $\{\omega\} \in \Omega^{even}(M;\CC)/\im d$
\[
S \circ \check a (\{\omega\}) = S \circ \imath(\{\omega\}) = S\circ \jmath\circ \sS^{-1}(\{\omega\}) = d\omega.
\]
At last we have the following
\begin{theorem}\label{odd diff K}
The functor $M\mapsto \dK(M)$ with the natural transformations $S$, $\check I$ and $\check a$ as above defines a differential extension of odd $K$-theory.
\end{theorem}

It is important in the sequel to have a thorough understanding of the action of forms map on $\dK$.
For any $\omega\in\Omega^{even}(M;\CC)$ we have $a(\{\omega\}) = \bm{\sE} - \ul{\bm{\sL\CC}}^n$, where $\bm{\sE} \in \struct_0(M)$ has the following property. For any smooth path $\c$ from a Flat pair $(\de,\d)$ to a chosen representative $(\triangle,\phi)$ of the string datum of $\bm{\sE}$, we have
\[
S(\c) = \omega \mod \widehat\wedge_{\Omega GL}(M),
\]
so $S(\c) = \omega + G^\ast\widehat\tau + d\chi$ for some smooth $G \colon M \to  \Omega GL$ and odd form $\chi$ for any such path $\c$.  By Remark \ref{all flat} the term $G^\ast\widehat\tau$ can be offset  by perturbing to a different Flat pair $(\tilde \de,\tilde \d)$ using a straight line segment. Thus we conclude that we can arrange to have a path $\tilde \c$ originating at  $(\tilde \de,\tilde \d)$ and ending at $(\triangle,\phi)$ that satisfies
$$S(\tilde \c) = \omega \mod \mbox{exact}. $$

\subsection{Odd differential $K$-theory}
So far we have shown that $\dK$ is a differential extension of odd $K$-theory. As proved in \cite{BS2}, there are infinitely many inequivalent  differential extensions of odd $K$-theory. However, differential extensions equipped with $S^1$-integration map are unique up to unique isomorphism and such an extension is what we mean by a \emph{model} for differential $K$-theory.

To show that $\dK$ does indeed define odd differential $K$-theory, we first fix a model $(\widehat K^\bullet,\widehat{ch},\widehat{I},\widehat{a})$ for differential $K$-theory with $\sone$-integration $\widehat{\int_\sone}$. 
Recalling our notation $(\edK,\check{ch}, \check I,\check{a}) $ for the Simons--Sullivan model for even differential $K$-theory, by \cite[Theorem 3.10]{BS3} there is a \emph{unique} natural isomorphism $\Phi_0 \colon \edK \to \widehat K^0$ preserving all of the structure.

Using the caloron correspondence together with the map $\Phi_0$, we define a map
\begin{equation}
\label{eqn:ktheoryiso}
\Phi_1 \colon \dK(M) \lo \widehat K^{-1}(M),
\end{equation}
sending a formal difference of structured $\Omega$ vector bundles $\bm{\sE}  =(\sE,[\triangle,\phi])$ and $\bm{\sF} = (\sF,[\tilde \triangle,\tilde\phi])$ to 
\begin{equation}
\label{eqn:odddiffkmap}
\Phi_1\big (\bm{\sE} - \bm{\sF}\big) := \widehat {\int_\sone} \Phi_0 \big( (E,[\nabla]) - (F,[\tilde \nabla]) \big)
\end{equation}
where $(E,\nabla)$ and $(F,\tilde \nabla)$ are the caloron transformed bundles corresponding to some choice of representatives of the string data.
A priori the definition of $\Phi_1$ depends on this choice and there is no canonical representative in the general case, but nevertheless we have the following
\begin{proposition}
\label{prop:welldefined}
The map $\Phi_1$ is a well defined group homomorphism.
\end{proposition}
\begin{proof}
Take any pairs of representatives $(\triangle_0,\phi_0),(\triangle_1,\phi_1)$ for $[\triangle,\phi]$ and similarly $(\tilde\triangle_0,\tilde\phi_0),$ $(\tilde\triangle_1,\tilde\phi_1)$ for $[\tilde \triangle,\tilde \phi]$. Write $\nabla_0,\nabla_1$ and $\tilde\nabla_0,\tilde\nabla_1$ respectively for the corresponding caloron transformed connections.
Showing that $\Phi_1$ is well defined is then equivalent to showing that
\[
\widehat{\int_\sone} \Phi_0 \big( (E,[\nabla_0]) - (F,[\tilde \nabla_0]) \big) = \widehat{\int_\sone} \Phi_0 \big( (E,[\nabla_1]) - (F,[\tilde \nabla_1]) \big).
\]
Let $\c$ be the line segment from $(\triangle_0,\phi_0)$ to $(\triangle_1,\phi_1)$ and $\tilde \c$ the line segment from $(\tilde \t_0,\tilde \phi_0)$ to $(\tilde \t_1,\tilde\phi_1)$. By definition of string data we have that $S(\c)$ and $S(\tilde \c)$ are exact.
Under the caloron transform, $\c$, $\tilde\c$ give rise to smooth paths $\c^c$, $\tilde\c^c$  on $E$ and $F$ or, equivalently, framed connections $\nabla^{\c^c}$, $\tilde\nabla^{\tilde\c^c}$ on $E\x\I$ and $F\x\I$ respectively. 
By naturality of $\Phi_0$, the Homotopy Formula \cite[Theorem 2.6]{BS3} applied to the even differential $K$-class $\widehat x:=\Phi_0 \big( (E\x\I,[\nabla^{\c^c}]) - (F\x\I,[\tilde\nabla^{\tilde\c^c}])\big)$ yields
\[
\Phi_0 \big( (E,[\nabla_1]) - (F,[\tilde\nabla_1]) \big) - \Phi_0 \big( (E,[\nabla_0]) - (F,[\tilde\nabla_0]) \big)  = \widehat a \Bigg({\int_\I}\,\widehat{ch}\big(\widehat{x}\big) \Bigg).
\]
Since $\Phi_0$ respects the curvature maps we have
\[
\widehat{ch}\big(\widehat{x}\big) = \check{ch} \big( (E\x\I,[\nabla^{\c^c}]) - (F\x\I,[\tilde\nabla^{\tilde\c^c}])\big) = \Ch(\nabla^{\c^c}) - \Ch(\tilde\nabla^{\tilde\c^c})
\]
and hence
\[
{\int_\I}\,\widehat{ch}\big(\widehat{x}\big) = \int_0^1 \varsigma_t^\ast\imath_{\d_t} \big(\Ch(\nabla^{\c^c}) - \Ch(\tilde\nabla^{\tilde\c^c}) \big)\,dt = \CS(\c^c) - \CS(\tilde \c^c).
\]
Using the relationship \eqref{eqn:stringycs} between the string potentials and the Chern--Simons forms, after integrating over the fibre we obtain
\[
\widehat{\int_\sone} \widehat a \,\Bigg({\int_\I}\,\widehat{ch}\big(\widehat{x}\big) \Bigg) = \check{a}\,\Bigg({\int_\sone}\CS(\c^c) - \CS(\tilde \c^c) \,\Bigg) = \check{a}\big(S(\c) - S(\tilde\c)\big) = 0
\]
since $S(\c) - S(\tilde\c)$ is exact and action of forms $\check{a}$ vanishes on exact forms by definition. Is is not hard to verify that $\Phi_1$ is a homomorphism, so we have the result.
\end{proof}

We now examine how $\Phi_1$ behaves with respect to the natural transformations $S$, $\check I$ and $\check a$.

\medskip
\noindent{\bf Curvature:}
For $\bm{\sE} - \bm{\sF} \in\dK(M)$ as above, we have
\[
S\big(\bm{\sE} - \bm{\sF} \big) = s(\triangle,\phi) - s(\tilde\triangle,\tilde\phi) = {\int_\sone}\big(\Ch(\nabla) - \Ch(\tilde\nabla)\big)
\]
and also
\[
\widehat{ch}\circ\Phi_1  \big(\bm{\sE} - \bm{\sF} \big) = {\int_\sone} \check{ch}\big((E,[\nabla]) - (F,[\tilde\nabla])\big) = {\int_\sone}\big(\Ch(\nabla) - \Ch(\tilde\nabla)\big)
\]
so that $\widehat{ch}\circ\Phi_1 = S$ as required.

\medskip
\noindent{\bf Underlying class:}
Consider the map $\dK(M) \to K^0(M\x\sone)$ given by
\[
\bm{\sE} - \bm{\sF} \longmapsto E-F,
\]
noting that this coincides with the composition of $\check I \colon \dK(M) \to K^{-1}(M)$ with the pullback  $q^\ast$ induced by the quotient map $q \colon M\x\sone \to \Sigma M^+$.
We also have
\[
\widehat I\circ\Phi_1\big(\bm{\sE} - \bm{\sF}\big) = {\int_\sone} \check I\big((E,[\nabla])-(F,[\tilde\nabla])\big) = {\int_\sone}\big(E-F\big),
\]
where the integration operations appearing in this expression are the $\sone$-integration maps on ordinary $K$-theory.
Recall the splitting $K^0(M\x\sone) \cong \im\pr^\ast\oplus\ker\imath^\ast$, with $\pr \colon M\x\sone \to M$ the projection and $\imath \colon m\mapsto (m,0)$ the canonical embedding, and also that the pullback $q^\ast$ is an isomorphism $K^{-1}(M)\to \ker\imath^\ast$ (the suspension isomorphism).
The $\sone$-integration on $K$-theory is defined as the composition of the projection onto $\ker\imath^\ast$ with the map $(q^\ast)^{-1}$. From this it follows that ${\int_\sone} \circ q^\ast = \id$ on the image of $\check I$ so that $\widehat I \circ\Phi_1 = \check I$ as required.

\medskip
\noindent{\bf Action of forms:}
Recall the characterisation of  $a(\{\omega\}) = \bm{\sE} - \ul{\bm{\sL\CC}}^n$ following Theorem \ref{odd diff K}. Given a structured $\Omega$ bundle $\bm{\sE} = (\sE,[\triangle,\phi])$, choose a representative $(\triangle,\phi)$ for the string datum and let $(E,\nabla)$ denote the caloron transform, then
\[
\Phi_1\circ a\,\big(\{\omega\}\big) =\widehat{\int_\sone} \Phi_0\big( (E,[\nabla]) - \ul{\bm{n}}\,\big).
\] 
Moreover, if $\c$ is a path from some Flat pair on $\sE$ to the chosen representative $(\triangle,\phi)$ such that $S(\c) = \omega$ mod exact, then 
\[
S(\c) = {\int_\sone} \CS(\c^c)  = {\int_\sone} \big(\tfrac{1}{2\pi} \omega \wedge d \theta\big),
\]
where $\c^c$ is the path on $E$ corresponding to $\c$ via the caloron transform.
The action of forms map $\check a$ on the Simons--Sullivan model $\edK(M)$ is defined in essentially the same fashion as above \cite[Section 3]{SSvec},  so we have that
\[
(E,[\nabla]) - \ul{\bm{n}} = \check a \big(\{\tfrac{1}{2\pi}\omega\wedge d\theta +\alpha\}\big)
\]
for some $\alpha$ in the kernel of ${\int_\sone}$.
Hence
\[
\widehat{\int_\sone} \Phi_0\big( (E,[\nabla]) - \ul{\bm{n}}\,\big) = \widehat{\int_\sone} \widehat a \big( \{\tfrac{1}{2\pi}\omega\wedge d\theta +\alpha\}\big) = \widehat a \big( \{\omega\}\big)
\]
so that $\widehat a = \Phi_1 \circ \check a$ as required. We now have the following
\begin{theorem}
\label{theorem:odddiffk}
The map $\Phi_1$ is an isomorphism.
\end{theorem}
\begin{proof}
The result is obtained by applying the five-lemma to the commuting diagram
\[
\xy0;/r.20pc/:
(0,25)*+{K^0(M)}="1";
(35,25)*+{\Omega^{even}(M;\CC)}="2";
(70,25)*+{\dK(M)}="3";
(105,25)*+{K^{-1}(M)}="4";
(140,25)*+{0}="5";
(0,0)*+{K^0(M)}="7";
(35,0)*+{\Omega^{even}(M;\CC)}="8";
(70,0)*+{\widehat K^{-1}(M)}="9";
(105,0)*+{K^{-1}(M)}="10";
(140,0)*+{0}="11";
{\ar^{ch} "1";"2"};
{\ar^{\check a} "2";"3"};
{\ar^{\check I} "3";"4"};
{\ar^{} "4";"5"};
{\ar^{} "1";"7"};
{\ar^{ch} "7";"8"};
{\ar^{\widehat a} "8";"9"};
{\ar^{\widehat I} "9";"10"};
{\ar^{} "10";"11"};
{\ar^{} "2";"8"};
{\ar^{\Phi_1} "3";"9"};
{\ar^{} "4";"10"};
{\ar^{} "5";"11"};
\endxy
\]
where the rows are exact and all unlabelled vertical arrows are the identity.
\end{proof}

\begin{remark}
The isomorphism $\Phi_1$ does not depend on the choice of model $\widehat K^\bullet$. Indeed, if $\widehat K^\bullet$ and $\widehat L^\bullet$ are any two models for differential $K$-theory, let $\Phi_0 \colon \edK \to \widehat K^0$ and $\Phi'_0 \colon \edK \to \widehat L^0$  be the corresponding unique natural isomorphisms from the Simons-Sullivan model.
There is also a unique natural isomorphism $\Psi \colon \widehat K^\bullet \to \widehat L^\bullet$ preserving all the structure, and by uniqueness we have $\Phi'_0 = \Psi \circ \Phi_0$.
Since $\Psi$ preserves the $\sone$-integrations on $\widehat K^\bullet$ and $\widehat L^\bullet$, it follows that if $\Phi_1 \colon \dK \to \widehat K^{-1}$ and $\Phi'_1 \colon \dK \to \widehat L^{-1}$ are the maps given by \eqref{eqn:odddiffkmap} then $\Phi'_1 = \Psi \circ \Phi_1$.
\end{remark}

The content of Theorem \ref{theorem:odddiffk} is essentially that $\dK$ defines odd differential $K$-theory and, together with the above remark, we have a \emph{canonical} isomorphism from $\dK$ to any other model for odd differential $K$-theory.
The results of Bunke--Schick tell us that when we have an $\sone$-integration map we may impose a uniqueness condition on this isomorphism.
As we now show, the inverse caloron transform induces a partial $\sone$-integration map
\[
{\int_\sone} \colon \edK(M\x\sone) \lo \dK(M)
\]
and the isomorphism appearing in Theorem \ref{theorem:odddiffk} is the \emph{unique} isomorphism respecting this operation. 
We begin by defining the partial $\sone$-integration. Just as for ordinary $K$-theory we have a splitting
\[
\edK(M\x\sone) \cong \im\pr^\ast\oplus\ker\imath^\ast.
\]
We emphasise that elements of $\ker\imath^\ast$ are of the form $\bm{E}- \ul{\bm{n}}$ where $\imath^\ast\bm{E}$ is stably flat, so by adding trivial structured vector bundles as necessary every element of $\ker\imath^\ast$ may be expressed in the form $(E',[\nabla]) - \ul{\bm{n}}$, where $E' \to M\x\sone$ is framed over $M_0$ and $\nabla$ is a framed connection.
The inverse caloron transform functor induces a well defined surjective homomorphism
$ \ker\imath^\ast \lo \dK(M)$ by sending
\[
(E,[\nabla]) - \ul{\bm{n}} \longmapsto (\sE,[\triangle,\phi]) - \ul{\bm{\sL\CC}}^n
\]
where $(\sE,\triangle,\phi)$ is the inverse caloron transform of $(E,\nabla)$.
We define the partial $\sone$-integration map
\[
{\int_\sone} \colon \edK(M\x\sone) \lo \dK(M)
\]
as the composition of the inverse caloron transform with the projection onto $\ker\imath^\ast$.
By construction ${\int_\sone}$ is natural, it vanishes on the image of $\pr^\ast$ and is compatible with the curvature and underlying class maps, that is
\[
{\int_\sone} \circ \check{ch} = S \circ {\int_\sone}
\;\;\mbox{and}\;\;
{\int_\sone} \circ \check I = \check I \circ {\int_\sone}.
\]
We also have the commuting diagram
\[
\xy
(0,25)*+{\edK(M\x\sone)}="1";
(35,25)*+{\widehat K^0(M\x\sone)}="2";
(0,0)*+{\dK(M)}="3";
(35,0)*+{\widehat K^{-1}(M)}="4";
{\ar^{\Phi_0} "1";"2"};
{\ar^{\widehat{\int_\sone}} "2";"4"};
{\ar^{{\int_\sone}} "1";"3"};
{\ar^{\Phi_1} "3";"4"};
\endxy
\]
which implies that ${\int_\sone}$ has all the properties required of an $\sone$-integration on $\edK$.
\begin{proposition}
The map $\Phi_1$ is the unique natural isomorphism $\dK \to \widehat K^{-1}$ respecting the integration map ${\int_\sone}\colon \edK(\cdot \x\sone) \to \dK(\cdot )$ defined via the inverse caloron transform.
\end{proposition}
\begin{proof}
Suppose that $\Psi \colon \dK \to \widehat K^{-1}$ is another such natural isomorphism.
Then we have
\[
\Phi_1 \circ {\int_\sone} = \widehat{\int_\sone}\circ \Phi_0 = \Psi \circ {\int_\sone}
\]
which implies that $\Phi_1 = \Psi$ since ${\int_\sone}$ is clearly surjective.
\end{proof}

We make a final remark on the odd differential Chern character on $\dK$.
The even differential Chern character in the Simons--Sullivan model is the natural map
\[
\check{\Ch} \colon \edK(M) \lo \check H^{even}(M) 
\]
mapping into even differential characters that sends $(E,[\nabla])- (F,[\nabla'])$ to the even differential character on $M$ defined by the difference $\CS(\nabla) - \CS(\nabla')$ of Chern--Simons forms.
Using this, we may define the odd differential Chern character similarly to the map $\Phi_1$ of Theorem \ref{theorem:odddiffk}.
That is, we send $\bm{\sE} - \bm{\sF}$ to the odd differential character
\[
\check{\Ch} \big (\bm{\sE} - \bm{\sF}\big) := {\int_\sone}\check{\Ch} \big( (E,[\nabla]) - (F,[\tilde \nabla]) \big)
\]
where $(E,\nabla)$ and $(F,\tilde \nabla)$ are the caloron transformed bundles corresponding to some choice of representatives of the string data.
The proof that this is well-defined follows along the same lines of Proposition \ref{prop:welldefined} using the Homotopy Formula.

\subsection{Hermitian $\Omega$ vector bundles}
We note that the entire approach of the preceding sections applies immediately to Hermitian $\Omega$ vector bundles equipped with compatible module connections and Higgs fields.
In particular, the string potential exactness condition that defines string data leads to a Hermitian version of the semigroup $\struct$ (and its various sub-semigroups) which we shall denote by $\struct_\RR$.
Proofs of all the preceding results proceed by analogy, so writing $\dK_\RR(M)$ for the rank zero subgroup of $K(\struct_\RR(M))$ we have the following
\begin{theorem}
The functor $M\mapsto \dK_\RR(M)$ with the natural transformations $S$, $\check I$ and $\check a$ defines odd differential $K$-theory.
\end{theorem}

\section{The Tradler--Wilson--Zeinalian model}\label{TWZ}
In a recent paper \cite{TWZ}, Tradler, Wilson and Zeinalian construct an elementary differential extension of odd $K$-theory, which we shall refer to as the \emph{TWZ extension}.
As previously mentioned, differential extensions of odd $K$-theory are non-unique and it is not yet known that the TWZ extension is a model for odd differential $K$-theory.
This is the content of the present section: by constructing an explicit isomorphism between the $\Omega$ model and that of Tradler--Wilson--Zeinalian we complete the construction of \cite{TWZ}.

The TWZ extension refines the homotopy theoretic model  $K^{-1}(M) = [M,U]$ using an equivalence relation on the space of smooth maps $M \to U$ that is strictly finer than smooth homotopy.
We briefly recall the salient points of the construction.
\begin{definition}
Two smooth maps $g_0,g_1 \colon M \to U$ are \emph{$\CS$-equivalent}, denoted $g_0 \sim_{\CS} g_1$, if and only if there exists a smooth homotopy $G$ from $g_0$ to $g_1$ such that
\[
\sum_{j=0}^\infty \frac{-j!}{(2j)!}\bigg(\!-\!\frac{1}{2\pi i}\bigg)^{j+1} \int_0^1 \tr\big(g_t^{-1} \d_t g_t \cdot (g_t^{-1} d g_t)^{2j} \big) dt = 0 \mod\mbox{exact}
\]
where $g_t :=G(\ ,t) \colon M\to U$.
The set $\dL(M)$ of equivalence classes inherits an abelian group structure from the block sum operation on maps $M \to U$ and the assignment $M\mapsto \dL(M)$ defines a contravariant functor on the category of compact manifolds with corners.
\end{definition}

The TWZ extension $\dL$ is a differential extension of odd $K$-theory, so it comes with curvature, underlying class and action of forms maps.

\medskip
\noindent{\bf Curvature:} The curvature map $R\colon \dL(M) \to \Omega^{odd}_{d=0}(M)$ is given by its action on $\CS$-equivalence classes as
\[
[g]_{\CS}\longmapsto \sum_{j=0}^\infty \frac{-j!}{(2j+1)!}\bigg(\!-\!\frac{1}{2\pi i}\bigg)^{j+1} \tr\big((g^{-1}dg)^{2j+1}\big).
\]

\medskip
\noindent{\bf Underlying class:} The underlying class map $J\colon \dL(M) \to K^{-1}(M)$ is simply the map
\[
[g]_\CS \longmapsto [g],
\]
on $\CS$-equivalence classes, where the equivalence classes on the right hand side are taken modulo smooth homotopy.

\medskip
\noindent{\bf Action of forms:} The action of forms map $b\colon \Omega^{even}(M)/\im d \to \dL(M)$ is given by sending $\{\omega\}$ to the class $[g]_\CS$, where for any choice $g \in[g]_\CS$ there is a smooth homotopy $G$ from the constant map to the identity $\id$ to $g$ such that
\[
\sum_{j=0}^\infty \frac{-j!}{(2j)!}\bigg(\!-\!\frac{1}{2\pi i}\bigg)^{j+1} \int_0^1 \tr\big(g_t^{-1} \d_t g_t \cdot (g_t^{-1} d g_t)^{2j} \big) dt  = \omega \mod \mbox{exact},
\]
with $g_t := G(\ ,t)$ as before.

\medskip
Consider now the natural map
\[
\imath\colon \dL(M) \to \dK_\RR(M)
\]
that sends
\[
[g]_\CS \longmapsto g^\ast\bm{\sE}(n) - \ul{\bm{\sL\CC}}^n
\]
for some choice of representative $g \colon M\to U(n)$ of the $\CS$-equivalence class.
We observe that $\imath$ is well defined, for if $g_0\sim_\CS g_1$ then take any smooth homotopy $G$ from $g_0$ to $g_1$.
Since $M\x\I$ is compact, we may without loss of generality suppose that the image of $G$ is contained in some $U(k)$ and so by \eqref{eqn:homstringform} we have, modulo exact forms,
\begin{multline*}
\cS\big(g_0^\ast\triangle(k),g_0^\ast\phi(k);g_1^\ast\triangle(k),g_1^\ast\phi(k)\big) = \int_0^1 g_t^\ast\imath_{\d_t g_t}s\big(\t(k),\phi(k)\big)\,dt\\
= \int_0^1 g_t^\ast\imath_{\d_t g_t} \sum_{j=0}^\infty \frac{-j!}{(2j+1)!}\bigg(\!-\!\frac{1}{2\pi i}\bigg)^{j+1} \tr\big(\Theta^{2j+1}\big) \,dt\\
= \sum_{j=0}^\infty \frac{-j!}{(2j)!}\bigg(\!-\!\frac{1}{2\pi i}\bigg)^{j+1} \tr\big(g_t^{-1}\d g_t\cdot (g_t^{-1}dg_t)^{2j}\big)\,dt = 0
\end{multline*}
where $g_0,g_1$ are viewed as maps $M \to U(k)$ by taking the block sum with the constant map at the identity as necessary.
It is easy to see that $R = S\circ \imath$ and $J =\check I\circ \imath$ so that $\imath$ preserves the curvature and underlying class maps.
Moreover by comparing $b$ with the characterisation of the action of forms map $\check a$ on $\dK$ following Theorem \ref{odd diff K}, we see that $\imath\circ b = \check a$, so that $\imath$ is a natural transformation of differential extensions. 
\begin{theorem}
The map $\imath$ is an isomorphism.
\end{theorem}
\begin{proof}
The proof is essentially identical to that of Theorem \ref{theorem:odddiffk}.
\end{proof}
\begin{remark}
This result demonstrates that every element of $\dK_\RR(M)$ may be written in the form $g^\ast\bm{\sE}(n) - \ul{\bm{\sL\CC}}^n$ for some $g\colon M\to U(n)$.
In particular, it implies that the canonical structured bundles $\bm{\sE}(n)$ over $U(n)$ classify $\dK_\RR$ and it also gives a homotopy theoretic interpretation of the Hermitian $\Omega$ model.
\end{remark}

\section*{Acknowledgements}
VSS acknowledges the support of a University of Adelaide Master of Philosophy Scholarship.  PH and VSS thank the Galileo Galilei Institute for Theoretical Physics for their hospitality and the INFN for partial support during the completion of this work. PH and RFV acknowledge the support of International Centre for Theoretical Physics, Trieste, and the Erwin Schr\"odinger Institute, Vienna. The authors would also like to thank Steve Rosenberg for helpful discussions and the referee for various useful suggestions that improved the clarity and quality of the paper.

\bibliographystyle{plain}

\end{document}